\documentclass[sts,a4paper,preprint,reqno,11pt]{my_imsart}

\RequirePackage[OT1]{fontenc}
\usepackage{amsthm,amsmath}
\RequirePackage[colorlinks,citecolor=blue,urlcolor=blue,breaklinks=true]{hyperref}

\usepackage{amssymb,a4wide}
\usepackage{amsfonts,mathrsfs}
\usepackage{url}
\usepackage{pifont,dsfont}
\usepackage{breakcites,microtype}
\usepackage[round]{natbib}

\usepackage{cleveref}
\usepackage{autonum}

\startlocaldefs

\newtheorem*{theorem*}{Theorem}

\newtheorem{remark}{Remark}[section]

\newtheorem{lemma}{Lemma}
\newtheorem{cor}{Corollary}

\newtheorem{proposition}{Proposition}

\newtheorem{theorem}{Theorem}

\newcommand{\R}{\mathbb R}
\newcommand{\RR}{\mathbb R}
\newcommand{\prob}{\mathbf{P}}
\newcommand{\esp}{\mathbb{E}}

\newcommand{\e}{\varepsilon}

\newcommand{\bfA}{\mathbf{A}}

\newcommand{\bfI}{\mathbf{I}}

\newcommand{\bfX}{\mathbf{X}}

\newcommand{\Xlab}{\bfX_{\rm lab}}

\newcommand{\bzero}{\mathbf{0}}

\newcommand{\bs}{\boldsymbol}
\newcommand{\bX}{\bs X}
\newcommand{\bY}{\bs Y}
\newcommand{\bZ}{\bs Z}

\newcommand{\bu}{\bs u}
\newcommand{\bv}{\bs v}
\newcommand{\bx}{\bs x}

\newcommand{\bbeta}{\bs\beta}
\newcommand{\bzeta}{\bs\zeta}

\newcommand{\bmu}{\bs\mu}

\newcommand{\bfSigma}{\mathbf \Sigma}

\newcommand{\excessRisk}{\mathcal E}
\newcommand{\transRisk}{\mathcal E_{\text{\rm TL}}}

\newcommand{\riskr}{\mathcal R}
\newcommand{\red}{\color{red}}
\newcommand{\blue}{\color{blue}}

\newcommand{\pen}{\text{\rm pen}}

\def\hat{\widehat}

\makeatletter
\let\orgdescriptionlabel\descriptionlabel
\renewcommand*{\descriptionlabel}[1]{%
    \let\orglabel\label
    \let\label\@gobble
    \phantomsection
    \protected@edef\@currentlabel{#1}
    \let\label\orglabel
    \orgdescriptionlabel{#1}%
}
\makeatother

\endlocaldefs

\begin{document}

\begin{frontmatter}

\title{On the prediction loss of the lasso in the partially labeled setting}
\runtitle{Lasso in the partially labeled setting}

\begin{aug}
\author{\fnms{Pierre C.} \snm{Bellec,}\ead[label=e1]{pierre.bellec,arnak.dalalyan,edwin.grappin@ensae.fr}}
\author{\fnms{Arnak S.} \snm{Dalalyan,}\ead[label=e2]{arnak.dalalyan@ensae.fr}}
\author{\fnms{Edwin} \snm{Grappin}\ead[label=e3]{edwin.grappin@ensae.fr}}
\and
\author{\fnms{Quentin} \snm{Paris}\ead[label=e4]{qparis@hse.ru}}

\runauthor{P. Bellec et al.}

\affiliation{ENSAE ParisTech - CREST and National Research University - Higher School of Economics}
\address{3 avenue Pierre Larousse,
92245 Malakoff, France.\hfill\break
26 Shabolovka street, Laboratory of Stochastic Analysis and its Applications,
Moscow, Russian Federation.}
\end{aug}

\begin{abstract}
In this paper we revisit the risk bounds of the lasso estimator in the context of transductive and semi-supervised
learning. In other terms, the setting under consideration is that of regression with random design under partial labeling.
The main goal is to obtain user-friendly bounds on the off-sample prediction risk. To this end, the simple setting of
bounded response variable and bounded (high-dimensional) covariates is considered. We propose some new adaptations of the lasso to
these settings and establish oracle inequalities both in expectation and in deviation. These results provide non-asymptotic
upper bounds on the risk that highlight the interplay between the bias due to the mis-specification of the linear model,
the bias due to the approximate sparsity and the variance. They also demonstrate that the presence of a large number
of unlabeled features may have significant positive impact in the situations where the restricted eigenvalue of the
design matrix vanishes or is very small.
\end{abstract}



\end{frontmatter}


\section{Introduction}
We consider the problem of prediction under the quadratic loss. That is, for a random feature-label pair
$(\bX,Y)$ drawn from a distribution $P$ on a product space $\mathcal X\times\mathcal Y$, we aim at predicting $Y$ as a function of $\bX$. The goal
is to find a measurable function $f:\mathcal X\to \mathcal Y$ such that the expected quadratic risk,
\begin{equation}
   \riskr(f)=\int_{\mathcal X\times\mathcal Y} (y-f(\bx))^2\,
    P(d\bx,dy)=\esp\big[\big(Y-f(\bX)\big)^2\big]
    \label{risk-of-f}
\end{equation}
is as small as possible. When $\mathcal Y$ is an interval of $\RR$  and $\mathcal X$ is a measurable set in $\RR^p$---which is the setting considered
in the present work---the Bayes predictor, defined as the minimizer of $\riskr(f)$ over all
measurable functions $f:\mathcal X\to\mathcal Y$, is the regression function \citep{Vapnik98}
\begin{equation}
    f^{\star}(\bx) = \esp[Y|\bX=\bx].
    \label{bayes}
\end{equation}
Using $f^\star$, the problem can be rewritten in a form which is more familiar in Statistics, namely
\begin{equation}
    \label{linearmodel}
    Y=f^{\star}(\bX)+\xi,
\end{equation}
where the noise variable $\xi$ satisfies $\esp[\xi|\bX]=0$, $P_X$-almost surely\footnote{Notation $P_X$ is used for
the marginal distribution of $\bX$.}. In the present work, we tackle the prediction problem in the case where the
available data $\mathcal D_{\rm all}$ is of the form $\mathcal D_{\rm all}=\mathcal D_{\rm labeled}\cup\mathcal D_{\rm unlabeled}$, where
$$
\mathcal D_{\rm labeled}=
\{(\bX_{1},Y_{1}),\dots,(\bX_{n},Y_{n})\}\quad\mbox{and}\quad\mathcal D_{\rm unlabeled}=\{\bX_{n+1},\dots,\bX_{N}\}.
$$
The labeled sample $\mathcal D_{\rm labeled}$ is composed of independent and identically distributed (i.i.d.) feature-label
pairs with distribution $P$. The unlabeled sample $\mathcal D_{\rm unlabeled}$ contains only i.i.d. features, with distribution $P_X$,
and is independent of $\mathcal D_{\rm labeled}$. This formal setting accounts for a number of realistic situations in which the
labeling process is costly while the unlabeled data points are available in abundance \citep[see, for instance,][]{FreeFoodCam05,GuillauminVS10,BrouarddS11},
that is $n$ may be quite small compared to $N$.
Here, the baseline idea is to build upon the sample $\mathcal D_{\rm unlabeled}$ to improve the supervised prediction process
based on $\mathcal D_{\rm labeled}$ alone. In this context, our study encompasses two closely related settings: semi-supervised
learning and transductive learning.

In the \textit{semi-supervised learning} setting, one aims at constructing a predictor $\hat f$, based on the data $\mathcal D_{\rm all}$,
such that the excess risk
\begin{equation}
    \label{excess}
    \excessRisk(\hat f)
    =
    \riskr(\hat f)-\riskr(f^{\star})= \int_{\R^p}\big(\hat f(\bx)-f^{\star}(\bx)\big)^2P_X(d\bx)=\|\hat f-f^\star\|_{L_2(P_X)}^2
\end{equation}
is as small as possible. This learning framework differs from the classical supervised learning only in that the data set is enriched by
the unlabeled features.

In contrast with this, the goal of \textit{transductive learning} is to predict solely the labels of the observed unlabeled features.
This amounts to considering the same setting as above but to measure the quality of a prediction function $f$
by the excess risk
\begin{equation}\label{excess1}
\transRisk(f)
=
\frac1{N-n}\sum_{i=n+1}^N
\big(f(\bX_i)-f^{\star}(\bX_i)\big)^2.
\end{equation}
We refer the reader to \citep{CSZ06,Zhu08} and the references therein for a comprehensive survey
on the topic of semi-supervised and transductive learning. Theoretical analysis of the generalisation
error and the excess risk in this context can be found in~\citep{Rigollet07,Wang07,lafferty2007}, whereas
the closely related area of manifold learning is studied in \citep{Belkin06,Nadler2009,Niyogi13}.
The purpose of the present work differs from these papers in that we put the emphasis on the high-dimensional
setting and the sparsity assumption. The goal is to understand whether the unlabeled data can help in predicting
the unknown labels using the $\ell_1$-penalized empirical risk minimizers. From another perspective---that of
multi-view learning---the problem of sparse semi-supervised learning is investigated in \citep{Shiliang10}.

When the feature vector is high dimensional, it is reasonable to consider prediction strategies
based on ``simple'' functions $f$ in order to limit the computational cost. A widely used approach is then
to look for a good linear predictor
\begin{equation}
    f_{\bbeta}(\bx)=\bx^\top\bbeta,\qquad \bbeta\in\R^p.
\end{equation}
When the dimension $p$ is of the same order as (or larger than) the size $n$ of the labeled sample, the
simple empirical risk minimizer (\textit{i.e.}, the least squares estimator) is a poor predictor since it suffers from the curse of
dimensionality. To circumvent this shortcoming, one popular approach is to use the $\ell_1$-penalised empirical
risk minimizer, also known as the lasso estimator \citep{Tib96}: $\hat f^{\rm lasso} = f_{\hat\bbeta{}^{\rm lasso}}$
where\footnote{To ease notation, we assume that both labels and features are centered, that is $\esp[Y]=0$ and $\esp[\bX]=0$,
so that there is no need to include an intercept in the linear combination $f_{\bbeta}$.}
\begin{equation} \label{lasso:original}
\hat\bbeta{}^{\rm lasso}\in\underset{\bbeta\in\R^p}{\arg\min}
\left\{\frac1n\|\bY-\Xlab\bbeta\|^2_2+2\lambda \|\bbeta\|_1 \right\},
\end{equation}
where $\lambda>0$ stands for a tuning parameter and
\begin{equation}
\label{Yx}
\bY=\begin{bmatrix}
Y_{1}\\
\vdots\\
Y_{n}
\end{bmatrix},\quad
\Xlab=\begin{bmatrix}
\bX^{\top}_{1}\\
\vdots\\
\bX^{\top}_{n}
\end{bmatrix}.
\end{equation}
Statistical properties of the lasso with regard to the prediction error
were studied in many papers, the most relevant (to our purposes) of which  will be discussed in the next section.
We also refer the reader to \citep{bookBuhlman2011} for an overview of related topics. The rationale behind this
approach is that (a) the term $\frac1n\|\bY-\Xlab\bbeta\|^2_2-\esp[\xi^2]$ is an unbiased estimator of the excess risk
$\excessRisk(f_{\bbeta})$ and (b) the $\ell_1$-penalty term favors predictors $f_{\bbeta}$ defined via
a (nearly) sparse vector $\bbeta$.

The prediction rules we are going to analyze in the present work are suitable adaptations of the (supervised) lasso to
the semi-supervised and the transductive settings. More precisely, we consider the estimator
\begin{equation} \label{lasso}
\hat\bbeta\in\underset{\bbeta\in\R^p}{\arg\min}
\left\{\|\bfA\bbeta\|^2_2-\frac2n\bY^\top\Xlab\bbeta +2\lambda \|\bbeta\|_1 \right\},
\end{equation}
where $\lambda>0$ and $\bfA\in\R^{p\times p}$ are parameters to be chosen by the statistician.
This definition is based on the following observation. The unlabeled sample may be used to get an improved estimator of
the excess risk $\excessRisk(f_{\bbeta}) = \esp[f^\star(\bX)^2]-2\esp[Y\bX^\top]\bbeta+\bbeta^\top\bfSigma \bbeta$, where
$\bfSigma = \esp[\bX\bX^\top]$ is the $p\times p$ covariance matrix. Indeed, the population covariance matrix can be estimated using
both labeled and unlabeled data. A similar observation holds for the transductive excess risk $\transRisk(f_{\bbeta})$.

Denoting by $\hat\bfSigma_{\rm lab}$ the empirical covariance matrix based on the labeled sample, that is
$$
\hat\bfSigma_{\rm lab}=
\frac1n\sum_{i=1}^n \bX_i\bX_i^\top,
$$
one checks that the vector $\hat\bbeta$ coincides with the lasso estimator  \eqref{lasso:original} when $\bfA = \hat\bfSigma_{\rm lab}^{1/2}$.
If an unlabeled sample is available, the foregoing discussion suggests a different choice for the matrix $\bfA$. This choice depends on the
setting under consideration. Namely, defining the matrices
$$
\hat\bfSigma_{\rm all}=\frac1N\sum_{i=1}^N \bX_i\bX_i^\top
\qquad \text{ and }
\qquad
\hat\bfSigma_{\rm unlab}=\frac1{N-n}\sum_{i=n+1}^N \bX_i\bX_i^\top,
$$
we use $\bfA=\hat\bfSigma_{\rm all}^{1/2}$ and $\bfA=\hat\bfSigma_{\rm unlab}^{1/2}$ in the semi-supervised and
transductive settings, respectively.

The following two assumptions made on the probability distribution $P$ will be repeatedly used throughout this work.\vspace{-6pt}
\begin{description}
\item[(A1)\label{A1}]The random variables $Y$ and $\bX$ have zero mean and finite variance. Furthermore, all
        the coordinates $X^{j}$ of the random vector $\bX$ satisfy $\esp[(X^{j})^2]=1$.
\item[(A2)\label{A2}]The random variables $Y$ and $X^{j}$ are almost surely bounded. That is, there exist
        constants $B_Y$ and $B_X$ such that $\prob\big(|Y|\le B_Y; \max_{j\in[p]}|X^{j}|\le B_X\big)=1$.
\end{description}
\vspace{-6pt}

Assumption \ref{A1}  is fairly mild, since one can get close to it by centering and scaling the observed labels and features. For features, the centering and the scaling
may be performed using the sample mean and the sample variance computed over the whole data-set. It is however important to require this assumption, since its violation
may seriously affect the quality of the $\ell_1$-penalized least-squares estimator $\hat\bbeta$, unless the terms $|\beta_j|$ of the $\ell_1$-norm are weighted according to the
magnitude of the corresponding feature $X^{j}$. The second assumption is less crucial both for practical and theoretical purposes, given that its primary aim is to allow for
user-friendly, easy-to-interpret theoretical guarantees. In most situations, even if assumption \ref{A2} is violated, the predictor $f_{\hat\bbeta}$ does have a
fairly small prediction error rate.

The main contributions of the present work are:\vspace{-6pt}
\begin{itemize}
\item Review of the relevant recent literature on the off-sample performance of the lasso in the prediction problem.
\item Non-asymptotic bounds for the prediction error of the lasso in the semi-supervised and transductive settings that guarantee the fast rate under the restricted eigenvalue
condition. We did an effort for keeping the results easy to understand and to obtain small constants. These results are simple enough to be taught to graduate students.
\item Oracle inequalities in expectation for the prediction error of the lasso. To the best of our knowledge, such results were not available in the literature until
the very recent paper \citep{BLT16}.
\end{itemize}

\vspace{-4pt}

To give a foretaste of the results detailed in the rest of this work, let us state and briefly discuss a risk bound in the
semi-supervised setting (the complete form of the result is provided in \Cref{thssms}). For a matrix $\bfA$, we denote by
$\|\bfA\|$ its largest singular value and by $\kappa_\bfA$ the compatibility constant (see \Cref{sec:2} for a precise definition).

\begin{theorem*}
    Let assumption \ref{A1} be fulfilled and let the random variables $Y$, $X^j$ be bounded in absolute value by 1.
    For a prescribed tolerance level  $\delta\in(0,1)$, assume that the overall sample size $N$ and the tuning parameter $\lambda$
    satisfy $N\ge 18p\Vert \bfSigma^{-1} \Vert\log({3p}/{\delta})$ and
    \begin{equation}
        \lambda \ge 4\left(\frac{2\log(6p/\delta)}{n}\right)^{1/2}+\frac{8\log(6p/\delta)}{3n}.
    \end{equation}
    Then, for every $J\subseteq \{1,\ldots,p\}$, with probability at least $1-\delta$, the estimator $\hat\bbeta$ defined in \eqref{lasso}
    above with $\bfA=\hat\bfSigma_{\rm all}^{1/2}$ satisfies
    \begin{align}\label{in:0}
    \excessRisk (f_{\hat\bbeta}) \le
    \inf_{\bbeta\in\R^p}
    \bigg\{
        \excessRisk (f_{\bbeta})
        + 4\lambda\|\bbeta_{J^c}\|_1
        + \frac{9\lambda^2|J|}{2\kappa_{\hat\bfSigma_{\rm all}}(J,3)}
    \bigg\}.
    \end{align}
\end{theorem*}
This result follows in the footsteps of many recent papers such as \citep{KLT11,Sun12,DHL14} among others. The term oracle inequality
refers to the fact that it allows us to compare the excess risk of the predictor $f_{\hat\bbeta}$ to that of the best possible nearly
sparse prediction function. (By nearly sparse we understand here a vector $\bbeta$ such that for a set
$J\subseteq\{1,\ldots,p\}$ of small cardinality the entries of $\bbeta$ with indices in $J^c$ have small magnitude; that is $\|\bbeta_{J^c}\|_1 =
\sum_{j\not\in J} |\beta_j|$ is small.) Indeed, if we denote by $\bar\bbeta$ a nearly $s$-sparse vector in $\RR^p$ such that
the excess risk $\excessRisk(f_{\bar\bbeta})$ is small, then the aforestated risk bound is the sum of three terms having clear
interpretation. The first term, $\excessRisk(f_{\bar\bbeta})$, is a bias term due to the $s$-sparse linear approximation. The second term,
$\lambda\|\bar\bbeta_{J^c}\|$, is the bias due to approximate $s$-sparsity. (Note that it vanishes if $\bar\bbeta$ is exactly
$s$-sparse and $J$ is taken as its support.)
Finally, the third term measures the magnitude of the stochastic error. {\color{red} Assuming the compatibility constant to be bounded away from $0$}, this last term is
of the order $s\log(p)/n$, which is known to be optimal\footnote{More precisely, the optimal rate is $\frac{s\log(1+p/s)}{n}$,
which is of the same order as $\frac{s\log(p)}{n}$ for most values of $s$.} over all possible estimators \citep{Fei10,Raskutti11,Rigollet11,Rigollet12a}.

Inequality  \eqref{in:0} readily shows the advantage of using the unlabeled data: the compatibility constant involved in the last term of
the right hand side is computed for the overall covariance matrix. When the size of the labeled sample is small in regard to the dimension $p$,
the corresponding constant computed for $\hat\bfSigma_{\rm lab}$ may be very close (and even equal) to zero. This may downgrade the fast rate of the
original lasso to the slow rate $\|\bar\bbeta\|_1/\sqrt{n}$. Instead, if a large number of unlabeled features are used, it becomes more plausible
to assume that the compatibility constant is bounded away from zero. In relation with this, it is important to underline that the unlabeled
sample cannot help to improve the fast rate of convergence of the lasso, $s\log(p)/n$, which is optimal in the minimax
sense. The best we can hope to achieve using the unlabeled sample is the relaxation of the conditions guaranteeing the fast rate.
Another worthwhile remark is that the theorem stated above is valid when the size of the unlabeled sample is significantly larger than the
dimension $p$. Interestingly, this condition is not required for getting the analogous result in the transductive set-up.

The rest is as follows. In \Cref{sec:2}, we introduce the notations used throughout the paper. \Cref{sec:5} contains
a review of the relevant literature and discusses the relation of the previous work with our results. \Cref{sec:3}
presents risk bounds for the prediction error of the lasso in the transductive setting, whereas \Cref{sec:4} is
devoted to the analogous results in the semi-supervised setting. Conclusions are made in \Cref{sec:6}. The proofs are
postponed to \Cref{sec:7}.

\section{Notations}\label{sec:2}
In the sequel, for any integer $k$ we denote by $[k]$ the set $\{1,\ldots,k\}$. For any $q\in [1,+\infty]$ the notation $\Vert \bv\Vert_q$ refers
to the $\ell_q$-norm of a vector $\bv$ belonging to an Euclidean space $\R^k$ with arbitrary dimension $k$. Since there is no risk of confusion, we omit the dependence on $k$ in
the notation. For any square matrix $\bfA\in\R^{p\times p}$ we denote by $\bfA^+$ its Moore-Penrose pseudoinverse and by $\Vert \bfA\Vert$ its spectral norm defined by
\begin{equation}
\label{sn}
\Vert \bfA\Vert=\max_{\Vert \bv\Vert_2=1}\Vert \bfA\bv\Vert_2
\end{equation}
We use boldface italic letters for vectors and boldface letters for matrices. Throughout the manuscript, the index $j$ will be used for referring to $p$ features, whereas
the index $i$ will refer to the observations ($i\in [n]$ or $i\in [N]$). For any set of indices $J\subseteq [p]$ and any $\bbeta=(\bbeta_1,\dots,\bbeta_p)^{\top}\in\R^p$,
we define $\bbeta_J$ as the $p$-dimensional vector whose $j$-th coordinate equals $\bbeta_j$ if $j\in J$ and $0$ otherwise. We denote the cardinality of any $J\subseteq[p]$
by $\vert J\vert$. Also, we set ${\rm supp}(\bbeta)=\{j: \bbeta_j\ne0\}$. In particular, whenever $f^{\star}(\bx) = \bx^\top\bbeta^{\star}$, we set $J^{\star}={\rm supp}(\bbeta^{\star})$
and $s^{\star}=\vert J^{\star}\vert$. For $J\subseteq[p]$ and $c>0$, we introduce the compatibility constants
\begin{equation}
\label{Cc}
\kappa_\bfA(J,c)=\inf\bigg\{\frac{c^{2}\vert J\vert\ \Vert \bfA^{1/2}\bv\Vert_2^{2}}{(c\Vert \bv_{J}\Vert_{1}-
\Vert \bv_{J^{c}}\Vert_{1})^{2}}: \bv\in\R^{p},\Vert \bv_{J^{c}}\Vert_{1}<c\Vert \bv_{J}\Vert_{1}\bigg\}
\end{equation}
and
\begin{equation}
\label{barCc}
\bar\kappa_\bfA(J,c)=\inf\bigg\{\frac{\vert J\vert\ \Vert \bfA^{1/2}\bv\Vert_2^{2}}{\Vert \bv_{J}\Vert^{2}_{1}}:
\bv\in\R^{p},\Vert \bv_{J^{c}}\Vert_{1}<c\Vert \bv_{J}\Vert_{1}\bigg\}.
\end{equation}
One easily checks that these two constants are of the same order of magnitude
in the sense that
$$
\frac{\bar c^2}{(\bar c+c)^2}\kappa_\bfA(J,\bar c+c)
\le \bar\kappa_\bfA(J,c) \le \kappa_\bfA(J,c)
$$
for every $c,\bar c>0$. These constants are slightly larger\footnote{We recall
here that a larger compatibility constant provides a better risk bound.} than the restricted eigenvalues
\citep{BRT} defined by
$$
\kappa^{\rm RE}_\bfA(J,c)=\inf\big\{\Vert \bfA^{1/2}\bv\Vert_2^{2}:\,
\Vert \bv_{J^{c}}\Vert_{1}\le c\Vert \bv_{J}\Vert_{1}\ \text{and}\ \Vert \bv_{J}\Vert_{2}=1\big\}.
$$
For more details, we refer the reader to \cite{VandeGeerConditionLasso09}.


\section{Brief overview of related work}\label{sec:5}

The material of this paper builds on the shoulders of giants and this section aims at providing a unified overview of
some of the most relevant results in our setting, without having the ambition of being exhaustive.
For each of the selected papers, we will
discuss its strengths and limitations in relation with the results presented further in this work.

Some recent results, obtained in the context of matrix regression, can be specialized to our problem and should be
put in perspective with our contribution.  For instance, a large part of Chapter 9 in \citep{koltchinskii2011} is devoted
to the problem of assessing the off-sample excess risk of the trace-norm penalized empirical risk minimizer in the setting
of trace regression with random design. One can arguably consider that setting as an extension of the random design regression
problem by restricting attention to the set of diagonal matrices. Then the estimator studied in \citet{koltchinskii2011}
coincides with the lasso estimator \eqref{lasso:original}. With our notations, the main result of Chapter 9
in \citep{koltchinskii2011} reads as follows.

\begin{theorem}[Theorem 9.3 in \citealp{koltchinskii2011}]
\label{thm:Kol1} Assume that Assumptions \ref{A1} and \ref{A2} hold. Then there exist universal positive
constants $c_1$ and $c_2$ such that, if
$$
\lambda\ge c_1B_X \max \left\{ {\frac{B_Y \log{(2p/\delta)}}{n}} ,\Big(\frac{B_Y \log{(2p/\delta)}}{n}\Big)^{1/2}\right\}
$$
for some $\delta\in(0,1)$, the estimator \eqref{lasso:original} satisfies,
\begin{align}
\excessRisk(f_{\hat\bbeta})
& \leq \inf_{\bbeta \in \mathbb{R}^{p}} \bigg \{ 2 \excessRisk(f_{\bbeta})
+ c_2 \bigg [ \frac{\Vert \bbeta \Vert_0\lambda^2}{\bar\kappa_{\bfSigma}(\text{\rm supp}(\bbeta), 5) }
+  {\Big(\Vert \bbeta \Vert_1  \vee \frac{q(\lambda)}{\lambda}\Big)\!}^2\, \frac{\log(k/\delta)\log(n)}{n}+\frac1n\bigg] \bigg \},
\end{align}
with probability larger than $1-\delta$, where
$$k =  \log(n\vee p\vee B_Y) \vee |\log(2\lambda)| \vee 2\quad\mbox{and}\quad
q(\lambda)=\inf_{\bbeta\in\RR^p} \big(\excessRisk(f_{\bbeta})+2\lambda\|\bbeta\|_1\big).$$
\end{theorem}
This result can be briefly compared to the risk bound in \eqref{in:0}.
The main advantages of this result is that (a) it is established under much weaker assumptions on the boundedness of the random variables
$\bX$ and $Y$ than those of Assumption \ref{A2}, (b) it holds not only for the vector regression but also for matrix regression, (c) it contains
no restriction on the sample size  and (d) it involves the compatibility constant of the population covariance matrix $\bfSigma$. On the
negative side, the oracle inequality in \Cref{thm:Kol1} is not sharp since the factor in front of $\excessRisk(f_{\bbeta})$ is not equal to one and,
more importantly, the rate of convergence of the remainder term is sub-optimal in most situations. Indeed, if the best linear predictor corresponds
to an $s$-sparse vector the nonzero entries of which are of the same order, then the term $\|\bbeta\|_1^2\log(k/\delta)\log(n)/n$,
present in the right hand side, is of order $s^2\log(n)\log\log(n+p)/n$, whereas the remainder term in \eqref{in:0} is of smaller order
$s\log(p)/n$.

On a related note, \citet{KLT11} establish sharp oracle inequalities for the trace-norm penalized least-squares estimator in the problem of
matrix estimation and completion under low rank assumption. Using our notation, Theorem 2 in \citep{KLT11} yields the following result.

\begin{theorem}[\citealp{KLT11}]
    \label{thKLT}
    Assume that the matrix $\bfSigma = \esp[\bX\bX^\top]$ is known
    and let $\hat\bbeta$ be as in \eqref{lasso} with $\bfA = \bfSigma^{1/2}$. Suppose in addition that Assumption \ref{A2} holds and that, for $\delta\in(0,1)$,
    $$
    \lambda \ge 4 B_Y \Big(\frac{\log(p/\delta)}{n}\Big)^{1/2}\bigg[1 + \frac{B_X}{3} \Big(\frac{\log(p/\delta)}{n}\Big)^{1/2}\bigg].
    $$
    Then, with probability larger than $1-\delta$, we have
    \begin{equation}
        \excessRisk (f_{\hat\bbeta}) \le
        \inf_{J\subseteq [p]}
        \inf_{\bbeta\in\R^p}
        \bigg\{
            \excessRisk (f_{\bbeta})
            + 4\lambda\|\bbeta_{J^c}\|_1
            + \frac{9\lambda^2|J|}{4\kappa_{\bfSigma}(J,3)}
        \bigg\}.
        \label{soi:KLT}
    \end{equation}
\end{theorem}

The original result \cite[Theorem 2]{KLT11} is slightly different from the aforestated one. In particular, it is
expressed in terms of the restricted eigenvalue constant with respect to the population covariance matrix $\bfSigma$.
However, all these differences imply only  minor modifications in the proofs. \Cref{thKLT} is very similar to the risk
bounds that we establish in the present work, but has the obvious shortcoming of requiring the covariance
matrix $\bfSigma$ to be known. In fact, this corresponds to the situation in which infinitely many unlabeled feature vectors
$\bX_{n+1},\bX_{n+2},...$ are available, that is $N=+\infty$. To some extent, one of the
purposes of the present work is to provide risk bounds analogous to the result of \Cref{thKLT}
but valid for a broad range of values of $N$. Note that the choice of the
tuning parameter $\lambda$ advocated by all the aforementioned results is of the same order of
magnitude.

To the best of our knowledge, the only paper establishing risk bounds for a transductive version of the lasso
is \citep{Alquier2012}. In that paper, the authors considered the problem of transductive learning in a linear
model $Y = \bX^\top\bbeta^{\star}+\xi$ under the sparsity constraint. The estimator they studied is slightly
different from ours and is defined by
\begin{equation} \label{lasso:1}
\hat\bbeta\in\underset{\bbeta\in\R^p}{\arg\min}
\left\{\|\hat\bfSigma_{\rm unlab}^{1/2}\bbeta\|^2_2-\frac2n\bY^\top\Xlab\hat\bfSigma_{\rm lab}^+\hat\bfSigma_{\rm unlab}\bbeta +2\lambda \|\bbeta\|_1 \right\}.
\end{equation}
For the predictor $f_{\hat\bbeta}$ based on this estimator, the authors established the following risk bound.

\begin{theorem}[Theorems 4.3 and 4.4 in \citealp{Alquier2012}]
\label{thAlquier2012}
Assume that for some $\bbeta^{\star}\in\R^p$, the conditional distribution of $\xi:=Y-\bX^\top\bbeta^{\star}$ given $\bX$
is Gaussian $\mathcal N(0,\sigma^2)$. Let ${\mathscr E}_1$ be the event ``all the unlabeled features $\{\bX_{n+i}:i\in[N-n]\}$, belong to the linear
span of the labeled features $\{\bX_i:i\in[n]\}$'' and let $\delta\in(0,1)$. Denote by $a_{n,N,p}$ the harmonic mean of the diagonal entries of the matrix
$\hat\bfSigma_{\rm unlab}\hat\bfSigma^{+}_{\rm lab}\hat\bfSigma_{\rm unlab}$. Then the estimator \eqref{lasso:1} with $\lambda = \sigma\sqrt{(2/n)a_{n,N,p}\log(p/\delta)}$ satisfies
$$
\prob\bigg(\transRisk(f_{\hat\bbeta}) \le \frac{72\sigma^2 a_{n,N,p} }{\kappa_{\hat\bfSigma_{\rm unlab}}(J^{\star},3)}\cdot
\frac{s^{\star}\log(p/\delta)}{n}\ \bigg|\ \bfX_{\rm all}\bigg)\ge 1-\delta\quad\text{on}\quad\mathscr E_1.
$$
\end{theorem}

This result is close in spirit to the result that we establish in this work in the setting of transductive learning.
Note however that there are three main differences. First, we do not confine our study to the well-specified situation
in which the Bayes predictor is linear, $f^{\star}(\bx)=\bx^\top\bbeta^{\star}$ for every $\bx\in\R^p$, with a sparse
vector $\bbeta^{\star}$. Second, we avoid the unpleasant restriction that the unlabeled features are linear combinations
of labeled features. Third, we replace the factor $a_{n,N,p}$---which may be quite large---by a more tractable quantity.
This being said, the result of \cite{Alquier2012}---in contrast with our results---does not
require the unlabeled features to be drawn from the same distribution as the labeled features.

We also review a recent result from \citep{LecueMend16}. In that paper, the authors
consider the isotropic case $\bfSigma = \bfI_p$, where $\bfI_p$ stands for the $p\times p$ identity matrix, but impose only weak assumptions
on the moments of the noise. Translated to our notations, their result can be formulated as follows.

\begin{theorem}[Theorem 1.3 in \citealp{LecueMend16}] Let Assumption \ref{A2} be satisfied and let $\bfSigma = \bfI_p$. Let $f_{\bar\bbeta}$ be the
best linear approximation in $L^2(P_X)$ of the regression function $f^\star$, that is  $\bar\bbeta\in\text{\rm arg}\min_{\bbeta\in\RR^p} \excessRisk(f_{\bbeta})$.
Let $\delta\in(0,1)$ be a prescribed tolerance level. There are three constants $c_1(\delta)$, $c_2(\delta,B_X)$ and $c_3(\delta,B_X)$ such
that, if $\bar\bbeta$ is nearly $s$-sparse in the sense that\footnote{We denote by $|\bar\beta|_{(j)}$ the $j$-th largest value of the sequence
$|\bar\beta_1|,\ldots,|\bar\beta_p|$, so that $|\bar\beta|_{(1)}\ge\dots\ge |\bar\beta|_{(p)}$.}
$$
\sum_{j={s+1}}^p |\bar\beta|_{(j)} \le c_1(\delta) B_Y s\Big(\frac{\log(2p)}{n}\Big)^{1/2}
$$
and $\lambda$ is chosen by $\lambda = c_2(\delta,B_X) B_Y \big(\frac{\log(2p)}{n}\big)^{1/2}$, then with probability at least $1-\delta$ the
lasso estimator satisfies
\begin{equation}
\excessRisk(f_{\hat\bbeta}) \le \excessRisk(f_{\bar\bbeta}) + c_3(\delta,B_X)B_Y^2\; \frac{ s\log(2p)}{n}.
\end{equation}
\end{theorem}
The principal strength of this result is that it is valid under a very weak assumption on the tails of the noise, but it has the shortcoming
of requiring the minimizer of the excess risk to be nearly $s$-sparse with a quite precise upper bound on the authorized non-sparsity bias.
From this point of view, an upper bound of the form \eqref{in:0} provides more information on the robustness of the prediction rule with
respect to the model mis-specification.

The proofs of the results above assess the off-sample prediction error rate of the lasso by using direct arguments. An alternative
approach \citep[adopted, for example, in][]{RWY2010,koltchinskii2011,oliveira2013lower,rudelson2013reconstruction} consists in taking advantage
of the in-sample risk bounds in order to assess the off-sample excess risk. In short, by means of nowadays well-known techniques
\citep[developed in][for instance]{BRT,juditsky2011accuracy,bookBuhlman2011,BCW2014,DHL14} for a well-specified model\footnote{This means that for a sparse
vector $\bbeta^\star$, it holds that $f^\star = f_{\bbeta^\star}$.}, an upper bound on the in-sample risk,{\color{red}[saut de ligne]}
$$\frac1n \|\bfX_{\rm lab}(\hat\bbeta-\bbeta^\star)\|_2^2= \|\hat\bfSigma_{\rm lab}^{1/2}(\hat\bbeta-\bbeta^\star)\|_2^2,$$
is obtained along with proving that the vector $\hat\bbeta-\bbeta^\star$ belongs to the dimension-reduction cone appearing in the
definition of the compatibility constant. Then, using suitably chosen concentration arguments, it is shown that (with high
probability) the compatibility constant  $\kappa_{\hat\bfSigma_{\rm lab}}(J^\star,c)$ of the empirical covariance matrix
$\hat\bfSigma_{\rm lab}$ is lower bounded by a (multiple of a) compatibility constant $\kappa_{\bfSigma}(J^\star,c')$ of the population
covariance matrix, provided that the sparsity $s$ is of order $n/\log(p)$. The main conceptual differences between the aforementioned papers are
in the conditions on the random vectors $\bX_i$. In \citep{RWY2010}, it is assumed that the $\bX_i$'s are Gaussian.
In \citet{rudelson2013reconstruction} and Theorem 9.2 in \cite{koltchinskii2011},
sub-Gaussian and bounded designs are considered, whereas only a bounded moment condition is required in \citet{oliveira2013lower}.
We will not reproduce their results here because (a) they do not allow to account for the robustness to the model mis-specification
and, to a lesser extent, (b) the constants involved in the bounds are not explicit.


\section{Risk bounds in transductive setting}\label{sec:3}

We first consider the case of transductive learning.
From an intuitive point of view, this case is simpler than the case of semi-supervised learning since a prediction
needs to be carried out only for the features in $\mathcal D_{\rm unlabeled}$. Indeed, recall from \eqref{excess1}
that in this context, the excess risk of the linear predictor $f_{\bbeta}$ is defined by
$$
\transRisk(f_{\bbeta})=\frac1{N-n}\sum_{i=n+1}^N
\big(\bX^{\top}_i\bbeta-f^{\star}(\bX_i)\big)^2
$$
and the suitably adapted lasso estimator is given by  choosing  $\bfA = \hat\bfSigma^{1/2}_{\rm unlab}$ in \eqref{lasso}, that is
$$
\hat\bbeta\in\underset{\bbeta\in\R^p}{\arg\min}
\Big\{\|\hat\bfSigma^{1/2}_{\rm unlab}\bbeta\|^2_2-\frac2n\bY^\top\Xlab\bbeta +2\lambda \|\bbeta\|_1 \Big\}.
$$
Note here that the role of the term $\frac2n\bY^\top\Xlab$ is to estimate the term $\frac2{N-n}\sum_{i={n+1}}^N f^\star(\bX_i)\bX_i^\top$,
which appears after developing the square in the excess risk. Since the latter belongs to the image of the matrix $\bfX_{\rm unlab}$, one can
slightly improve the estimator by projecting onto the subspace of $\RR^p$ spanned by the unlabeled vectors $\bfX_i$. This amounts to replacing
the term $\bY^\top\Xlab\bbeta$ by $\bY^\top\Xlab\Pi_{\rm unlab}\bbeta$, where $\Pi_{\rm unlab}$ stands for the orthogonal projector in $\RR^p$
onto $\text{Span}(\bX_{n+1},\ldots,\bX_{N})$. However, from a theoretical point of view, this modification has no impact on the risk
bound stated below. That is why we confine our attention to the lasso estimator that does not use this modification.


\begin{theorem}
\label{thtransductive}
Let Assumptions \ref{A1} and \ref{A2} be fulfilled. Define $n_{\star}= n\wedge(N-n)$ and assume that, for a given $\delta\in (0,1)$, the tuning parameter
$\lambda$ satisfies
\begin{equation}
\lambda\ge 4B_Y\Big(\frac{\log(2p/\delta)}{n_{\star}}\Big)^{1/2}\bigg[1+\frac{B_X}{3}\Big(\frac{\log(2p/\delta)}{n_{\star}}\Big)^{1/2}\bigg].
\label{lambdatrans}
\end{equation}
Then, with probability at least $1-\delta$, the predictor $f_{\hat\bbeta}$ satisfies
\begin{align}
\transRisk(f_{\hat\bbeta}) \le  \inf_{\substack{\bbeta\in\R^p\\ J\subseteq[p]}}\bigg\{\transRisk(f_{\bbeta})
    + 4\lambda\|\bbeta_{J^c}\|_1+ \frac{9\lambda^2|J|}{4\kappa_{\hat\bfSigma_{\rm unlab}}(J,3)}\bigg\}.
    \label{thtrans}
    \end{align}
\end{theorem}

A few comments are in order. First, \Cref{thtransductive} holds for any pair of integers $n$ and $N$ larger than $1$.
However, it is especially relevant when the number $N-n$ of unlabeled features is larger than the number $n$ of labeled ones.
As already mentioned, this kind of situation is frequent in applications where the labeling procedure is expensive.
In this case, $n_{\star}=n$ and \Cref{thtransductive} takes the same form as \eqref{in:0} with the notable advantage that the
size of the unlabeled sample does not need to be of larger order than the dimension $p$. Let us present a few implications of
this result in the well-specified case.

\paragraph{Well-specified case.}
Recall that the well-specified case refers to the situation where there exists $\bbeta^\star\in\R^p$
such that the Bayes predictor $f^{\star}$ satisfies $f^{\star}(\bx)=\bx^{\top}\bbeta^{\star}$, $P_{X}$-almost surely.
In this case, the excess risk of a predictor $f_{\bbeta}$ can be written as
$
\transRisk(f_{\bbeta})=\Vert\hat\bfSigma^{1/2}_{\rm unlab}(\bbeta-\bbeta^\star)\Vert^2_2.
$
In this form, the technical tractability of the transductive learning problem appears clearly since the matrix
$\bfA=\hat\bfSigma^{1/2}_{\rm unlab}$ used in the definition of the estimator $\hat\bbeta$ coincides with the
one appearing in the excess loss. As we shall see later, this is indeed not the case for semi-supervised learning.
Now, the choice of $\bbeta=\bbeta^{\star}$ and $J=J^{\star}$ in the right hand side of inequality \eqref{thtrans}
yields
$$
\transRisk(f_{\hat\bbeta}) \le  \frac{9\lambda^2 s^{\star}}{4\kappa_{\hat\bfSigma_{\rm unlab}}(J^{\star},3)}.
$$
The choice of $\lambda$ provided by the right hand side of inequality \eqref{lambdatrans}, along with the condition
$n_{\star}\ge B_X^2\log(2p/\delta)$, leads to the bound
$$
\transRisk(f_{\hat\bbeta}) \le \frac{64B^{2}_{Y}}{\kappa_{\hat\bfSigma_{\rm unlab}}(J^{\star},3)}\cdot\frac{s^{\star}\log(p/\delta)}{n_{\star}},
$$
with probability at least $1-\delta$.
Comparing our result with that of \citet{Alquier2012} (cf. \Cref{thAlquier2012} above), we can note that \Cref{thtransductive}
holds without the assumption that the unlabeled features belong to the linear span of the labeled ones. On the other hand,
\citet{Alquier2012} do not require the labeled and the unlabeled features to be drawn from the same distribution.

\section{Risk bounds in semi-supervised setting}\label{sec:4}

We now turn to the more challenging problem of semi-supervised learning. In this subsection,
we first consider the well-specified setting in which the Bayes predictor $f^{\star}$ is linear.
We start with risk bounds that hold with a probability close to one. Such bounds are often termed
\textit{in deviation} as opposed to those holding \textit{in expectation}.
\paragraph{Well-specified case.}
We assume here that
\begin{equation}
\label{well-spec}
f^{\star}(\bx) = \bx^\top \bbeta^{\star},\quad P_X\text{-almost surely}.
\end{equation}
In this context, the excess risk of the linear predictor $f_{\bbeta}$, defined in \eqref{excess}, becomes
$\excessRisk(f_{\bbeta})=\|\bfSigma^{1/2}(\bbeta-\bbeta^\star)\|_2^2$. This setting is more restrictive than
the mis-specified setting considered below, but it has the advantage of allowing us to obtain risk bounds
that are small even if the sample size $N$ is not necessarily larger than the dimension $p$.
The next result assesses the performance of the predictor $f_{\hat\bbeta}$ where
\begin{equation}
    \label{eq:lasso-sem-sup}
    \hat\bbeta\in\underset{\bbeta\in\R^p}{\arg\min}
\left\{\|\hat\bfSigma^{1/2}_{\rm all}\bbeta\|^2_2-\frac2n\bY^\top\Xlab\bbeta +2\lambda \|\bbeta\|_1 \right\},
\end{equation}
corresponding to the choice $\bfA=\hat\bfSigma^{1/2}_{\rm all}$ in \eqref{lasso}.
In the next result, we set $$\kappa^{\rm RE}_\bfA(s,c) = \min_{J\subseteq [p] : |J|\le s} \kappa^{\rm RE}_\bfA(J,c),$$
where the restricted eigenvalue $\kappa^{\rm RE}_\bfA(J,c)$ is defined in \Cref{sec:2}.

\begin{theorem}
\label{thssws}
Let Assumptions \ref{A1}, \ref{A2} and \eqref{well-spec} be fulfilled. Let $\delta\in(0,1)$ be a tolerance level and
let the tuning parameter $\lambda$ satisfy
$$
\lambda\ge 4B_Y\Big(\frac{\log(4p/\delta)}{n}\Big)^{1/2}\bigg[1+\frac{B_X}{2}\Big(\frac{\log(4p/\delta)}{n}\Big)^{1/2}\bigg].
$$
With probability at least $1-\delta$, it holds
\begin{equation}\label{thm:6:1}
\excessRisk(f_{\hat\bbeta})\le \bigg(\frac{6\lambda s^{\star}}{\bar\kappa_{\hat\bfSigma_N}(J^{\star},3)}\bigg)^2\bigwedge
\frac{9\|\bfSigma\|\lambda^2s^\star}{\kappa^{\rm RE}_{\hat\bfSigma_N}(s^\star,3)^2}.
\end{equation}
In addition, if the overall sample size $N$ is such that $16s^{\star}B_X^2\sqrt{2\log(4p^2/\delta)}\le \bar\kappa_{\bfSigma}(J^{\star},3)\sqrt{N}$ then, with probability at least
$1-\delta$, the predictor $f_{\hat\bbeta}$ satisfies the inequality
\begin{align}
\excessRisk(f_{\hat\bbeta})\le \frac{9\lambda^2 s^{\star}}{\bar\kappa_{\bfSigma}(J^{\star},3)}.\label{thm:6:2}
\end{align}
\end{theorem}

This theorem provides three different risk bounds, all of them being valid for the same choice of the tuning parameter $\lambda$,
that clearly show the benefits of using unlabeled data. The first two bounds are stated in \cref{thm:6:1}. They share the common feature of
depending on a characteristic (compatibility constant or restricted eigenvalue) of the sample covariance matrix. The latter
is computed using both labeled and unlabeled data. For large values of $N$, it is more likely that these characteristics are
bounded away from zero than those of the sample covariance matrix based on the labeled data only. In the asymptotic setting
where $s^\star$ goes to infinity with the sample size and the dimension, the second term in the right hand side of \cref{thm:6:1} is
of smaller order than the first one and is rate optimal, provided that the restricted eigenvalue is lower bounded by a fixed positive
constant. However, for finite and small values of $s^\star$ the first term in the right hand side of \cref{thm:6:1} might be
smaller than the second term.

This being said, it might be more insightful to look at the non random upper bounds on the excess risk as the one stated in
\cref{thm:6:2}.  It basically tells us that if the overall sample size is larger than a multiple of $(s^{\star})^2\log p$, then the
off-sample prediction risk of the semi-supervised lasso estimator achieves the fast rate $\frac{s^{\star}\log p}{n}$.
Note that if we use only the labeled data points, the best known results---as recalled in \Cref{sec:2} above---provide the fast
rate when $n$ is larger than a multiple of $s^{\star}\log p$. Thus, if $N$ is of the same order as $n$, our result above is not the
sharpest possible, but it has the advantage of being easy to prove and, nevertheless, demonstrating the gain of using the
unlabeled data. In particular, the proof of results providing the fast rate under the condition $n\ge Cs^{\star}\log(p)$,
for some $C>0$, involve the important step of lower bounding the compatibility constant of the sample covariance
matrix by its population counterpart. This step uses concentration arguments which are often tedious and come with
implicit (or unreasonably large) constants. Instead, our proof makes use of much simpler tools essentially boiling
down to the classical Bernstein inequality and leads to explicit and small constants.

\paragraph{Mis-specified case.}
Mathematical analysis of the semi-supervised lasso under mis-specification is more involved, since it requires careful
control of the bias terms corresponding to the nonlinearity and the non-sparsity of the model. We first state results
providing risk bounds in deviation, then state their counterpart in expectation.

\begin{theorem}
\label{thssms}
    Let Assumptions \ref{A1} and \ref{A2} be fulfilled. Fix $J\subseteq [p]$ and $\delta\in(0,1)$. Suppose in addition that
    \begin{equation}
        \label{condition-N-p}
        N\ge  18 B^2_{X}p\Vert \bfSigma^{-1} \Vert\log({3p}/{\delta})
    \end{equation}
    and
    \begin{equation}
        \label{condition-ssms-lambda}
        \lambda \ge 8B_{X}B_Y\Big(\frac{\log(6p/\delta)}{n}\Big)^{1/2}\bigg[1+\frac{B_X}{3}\Big(\frac{\log(6p/\delta)}{n}\Big)^{1/2}\bigg].
    \end{equation}
	Then the semi-supervised lasso estimator $\hat\bbeta$ defined in \eqref{eq:lasso-sem-sup} above
    satisfies
    \begin{align}
    \excessRisk (f_{\hat\bbeta}) \le
    \inf_{\bbeta\in\R^p}
    \bigg\{
        \excessRisk (f_{\bbeta})
        + 4\lambda\|\bbeta_{J^c}\|_1
        + \frac{9\lambda^2|J|}{2\kappa_{\hat\bfSigma_{\rm all}}(J,3)}
    \bigg\},
    \label{eq:soi-semisup}
    \end{align}
    with probability larger than $1-\delta$.
\end{theorem}

The novelty of \Cref{thssms} lies in the semi-supervised nature of the estimator \eqref{eq:lasso-sem-sup}, which involves all
the unlabeled features through the matrix $\bfA=\hat\bfSigma^{1/2}_{\rm all}$ in \cref{lasso}. In particular, \Cref{thssms}
quantifies the natural intuition according to which, if $N$ is large enough, the matrix $\bfA=\hat\bfSigma^{1/2}_{\rm all}$
is a good estimator of $\bf\Sigma$ and a result similar to \Cref{thKLT} should hold. As mentioned in the introduction, an
attractive feature of the upper bound in \cref{eq:soi-semisup} is that it is of the same form as the recent oracle inequalities
established in the case of fixed design regression \citep[see, for instance,][and the references therein]{DHL14,Pensky14} and quantify in an
easy-to-understand manner the error terms accounting for the non-linearity and the non-sparsity of the true regression function $f^\star$.

The minimal number $N$ of features satisfying \eqref{condition-N-p} depends on $\Vert \bfSigma^{-1} \Vert=\lambda^{-1}_{\rm min}(\bf\Sigma)$,
reflecting the fact that the quality of approximation of the identity matrix $\bfI_p$ by
$\bfSigma^{-1/2}\hat\bfSigma_{\rm all}\bfSigma^{-1/2}$ depends on $\Vert \bfSigma^{-1} \Vert$. One can remark that under constraint \eqref{condition-N-p},
the lowest eigenvalue of the sample covariance matrix is close to its population counterpart \citep{Vershynin11} and provides a
simple lower bound on the compatibility constant  $\kappa_{\hat\bfSigma_{\rm all}}(J,3)$ appearing in \cref{eq:soi-semisup}. These
considerations lead to the following corollary.

\begin{cor}\label{cor:1}
    Under the conditions of \Cref{thssms}, with probability at least $1-\delta$,
    \begin{align}
    \excessRisk (f_{\hat\bbeta}) \le
    \inf_{J\subseteq [p]}
    \inf_{\bbeta\in\R^p}
    \Big\{
        \excessRisk (f_{\bbeta})
        + 4\lambda\|\bbeta_{J^c}\|_1
        + {\textstyle\frac{27\| \bfSigma^{-1} \|}{4}}\, \lambda^2 |J|
    \Big\}.
    \end{align}
\end{cor}

Let us also mention that the factor $B_X^2 p\|\bfSigma^{-1}\|$ present in the right hand side of \cref{condition-N-p}
is an upper bound on the norm $\|\bfSigma^{-1/2}\bX_i\|_2^2$ under assumption \ref{A2}. Under additional assumptions on the
support of the features $\bX_i$, this expression may be replaced by a smaller one leading thus to a relaxation of
condition~\eqref{condition-N-p}.

\paragraph{Sharp oracle inequality in expectation.}

All the previously stated results assert that the lasso estimator has a small prediction error
on an event of overwhelming probability. However, in these results, the choice of the tuning
parameter $\lambda$ and, therefore, the final predictor $f_{\hat\bbeta}$, depends on the prescribed
level of tolerance. A consequence of this dependence is that one can not integrate out the
bounds in deviation in order to get a bound in expectation. This is probably one of the reasons
why the bounds in expectation for the lasso are scarce in the literature. To fill this caveat,
we state below a risk bound in expectation that can be easily deduced from the bounds in deviation.

\begin{theorem}
\label{thssms1}
    Let Assumptions \ref{A1} and \ref{A2} be fulfilled. Suppose that
    the overall sample size is such that $N\ge  18 B^2_{X}p\Vert \bfSigma^{-1} \Vert\log(3pN^2)$. Then, for the tuning
		parameter
    \begin{equation}
        \label{condition-ssms-lambda1}
        \lambda = 8B_{X}B_Y\Big(\frac{\log(6pN^2)}{n}\Big)^{1/2}\bigg[1+\frac{B_X}{3}\Big(\frac{\log(6pN^2)}{n}\Big)^{1/2}\bigg]
    \end{equation}
	the semi-supervised lasso estimator $\hat\bbeta$ defined in \eqref{eq:lasso-sem-sup} above
    satisfies
    \begin{align}
    \esp[\excessRisk (f_{\hat\bbeta})] \le
    \inf_{J\subseteq[p]}\inf_{\bbeta\in\R^p}
    \Big\{
        \excessRisk (f_{\bbeta})
        + 4\lambda\|\bbeta_{J^c}\|_1
        + \frac{27\|\bfSigma^{-1}\|}{4} \lambda^2|J|
    \Big\} + \frac{2B_Y^2}{N^2}+\frac{B_Y^2}{2^7n\log^2(6pN^2)}.
    \label{eq:soi-semisup1}
    \end{align}
\end{theorem}

The proof of this theorem is postponed to \cref{ssec:74}. The bound above is not optimal in terms of its dependence on $N$.
In particular, it blows up when $N$ goes to infinity and all the other parameters are fixed. However, this divergence is
only logarithmic in $N$. The dominating term in the risk bound above is (at least in the well specified setting) of the order
$\lambda^2 |J| \asymp \frac{s\log(pN)}{n}$.

\section{Conclusion}\label{sec:6}

We have reviewed some recent results on the prediction accuracy of the lasso in the problem of regression with random design
and have proposed their extensions to the setting where the labels of some data points are not available. Theoretical
guarantees stated in previous sections are formulated as oracle inequalities that allow us to compare the excess risk of a
suitable adaptation of the lasso to the best possible (nearly) sparse prediction function. We have opted for considering only
those risk bounds that provide the fast rate and are valid under some conditions on the design such as the restricted eigenvalue
condition or the compatibility condition. Some of the established upper bounds involve the compatibility constant of the
sample covariance matrix. Using results on random matrices \citep{rudelson2013reconstruction,oliveira2013lower, bah14} they can
be further worked out to get deterministic upper bounds. However, the evaluation of the restricted eigenvalues and related
quantities of the random covariance-type matrices is a dynamically evolving research area and we expect that new advances
will be made in near future.

The main high level message of the contributions of this paper is that one can take advantage of the unlabeled sample for
improving the prediction accuracy of the lasso. Roughly speaking, if the size of the unlabeled sample is larger than the
ambient dimension, then the modified lasso predictor has a prediction risk that converges to zero at the optimal rate even
if the sample covariance matrix based only on the labeled sample does not satisfy the compatibility or the restricted
eigenvalue condition. However, it should be acknowledged that when the model is well specified (that is there exists a sparse
linear combination of the features with an extremely low approximation error) and the population covariance matrix is
well-conditioned, then the original lasso might perform as well as, or even better than, the modified lasso proposed in this
work. Therefore, one can conclude that the use of the unlabeled sample improves on the robustness of the lasso to the model
mis-specification.

We would like also to emphasize that, pursuing pedagogical goals, we have restricted our attention to the simple case
of bounded feature vectors and bounded labels. All the proofs presented in this paper are based on elementary arguments
and are fairly simple. Using more involved arguments, they can be carried over the case of sub-Gaussian design and labels.
It would be interesting to explore their extensions to other settings such as regression with structured
sparsity, low rank matrix regression or matrix completion, \textit{etc}.

\section{Proofs}\label{sec:7}

We start with a general result that holds for the penalized least squares predictor with arbitrary convex penalty.
This result is of independent interest. It generalizes the corresponding result of \citep{KLT11} established for the
matrix trace-norm penalties. The proof that we present here is different from the one in \citep{KLT11} in that it
does not rely on the precise form of the sub-differential of the penalty function.
\begin{lemma}\label{lem:1}
    Let $n,p\ge 1$.
Let $\pen:\R^p\to\R$ be any convex function and $\hat\bbeta$ be defined by
\begin{equation} \label{pen-ls}
\hat\bbeta\in\underset{\bbeta\in\R^p}{\arg\min}
\left\{\|\bfA\bbeta\|^2_2-\frac2n\bY^\top\Xlab\bbeta +\pen(\bbeta) \right\},
\end{equation}
where $\bfA\in\R^{p\times p}$, $\bY\in\R^n$ and $\Xlab\in\R^{n\times p}$. Then, for all $\bbeta\in\R^p$,
\begin{align}\|\bfA\hat\bbeta\|^2_2\le \|\bfA\bbeta\|^2_2+ \frac2n\bY^\top\Xlab(\hat\bbeta-\bbeta) +\pen(\bbeta)-\pen(\hat\bbeta)-\|\bfA(\hat\bbeta-\bbeta)\|_2^2.
\label{eq:1}
\end{align}
\end{lemma}
\begin{proof}
Let us introduce the function $\Phi(\bbeta) =  \|\bfA\bbeta\|^2_2-\frac2n\bY^\top\Xlab\bbeta +\pen(\bbeta)$ for every $\bbeta\in\R^p$, so that
$\hat\bbeta$ is a minimum point of $\Phi$. Since the latter is a convex function, we know that the zero vector $\bzero_p$ of $\R^p$ belongs to the sub-differential
$\partial \Phi(\hat\bbeta)$ of $\Phi$ at $\hat\bbeta$.
For all $\bbeta\in\R^p$, let
\begin{equation}
    \psi(\bbeta) = \|\bfA(\bbeta-\hat\bbeta)\|^2_2,
    \qquad
    \bar\Phi (\bbeta) = \Phi(\bbeta) - \psi(\bbeta).
\end{equation}
The function $\psi$ is proper and convex.
It is also differentiable on $\R^p$
and the sub-differential of $\psi$ at $\hat\bbeta$ is reduced to its gradient at $\hat\bbeta$, so that
$\partial\psi(\hat\bbeta) = \{ \nabla \psi (\hat\bbeta) \} = \{\bzero_p\}$.
The function $\bar\Phi$ defined on $\R^p$
is the sum of an affine function and the convex function $\pen$, thus it is also convex.
The functions $\psi,\bar\Phi$ are proper and convex,
the function $\psi$ is continuous on $\R^p$ so
by the Moreau-Rochafellar Theorem,
\begin{equation}
    \partial\Phi(\hat\bbeta)
    =
    \partial\psi(\hat\bbeta)
    +
    \partial\bar\Phi(\hat\bbeta)
    =
    \{\bzero_p\}
    +
    \partial\bar\Phi(\hat\bbeta)
    =
    \partial\bar\Phi(\hat\bbeta).
\end{equation}
Thus $\bzero_p\in\partial\bar\Phi(\hat\bbeta)$, which can be rewritten as
\begin{equation}
    \bar\Phi(\bbeta) \ge \bar\Phi(\hat\bbeta), \qquad \forall \bbeta\in\R^p.
\end{equation}
By adding $\psi(\bbeta)$ on both sides of the previous display, we obtain
\begin{equation}
\Phi(\bbeta) \ge \Phi(\hat\bbeta) + \|\bfA(\hat\bbeta-\bbeta)\|_2^2,\qquad \forall \bbeta\in\R^p.
\end{equation}
Rearranging the terms of this inequality, we get the claim of the lemma.
\end{proof}

We will also repeatedly use the following result.

\begin{lemma}\label{lem:2}
For any pair of vectors $\bbeta,\bbeta'\in \R^p$, for any pair of scalars $\mu>0$ and $\gamma>1$, for any $p\times p$ symmetric matrix $\bfA$ and for
any set $J\subseteq [p]$, the following inequality is true
\begin{align}
2\mu\gamma^{-1}\Big(\|\bbeta-\bbeta'\|_1+\gamma\|\bbeta\|_1-\gamma\|\bbeta'\|_1\Big) -\|\bfA(\bbeta-\bbeta')\|_2^2\le 4\mu\|\bbeta_{J^c}\|_1+ \frac{(\gamma+1)^2\mu^2|J|}{\gamma^2\kappa_{\bfA^2}(J,c_\gamma)},
\end{align}
where $c_\gamma = (\gamma+1)/(\gamma-1)$.
\end{lemma}
\begin{proof}
To ease notation, we set $\bu = \bbeta-\bbeta'$. Using that $\Vert\bbeta_J\Vert_1-\Vert\bbeta'_J\Vert_1\le\Vert\bu_J\Vert_1$ and
$\Vert\bbeta_{J^c}\Vert_1+\Vert\bbeta'_{J^c}\Vert_1\ge\Vert\bu_{J^c}\Vert_1$, we obtain
\begin{align}
\|\bu\|_1+\gamma\|\bbeta\|_1-\gamma\|\bbeta'\|_1
& =\|\bu\|_1+\gamma\big(\|\bbeta_J\|_1+\|\bbeta_{J^c}\|_1-\|\bbeta'_J\|_1-\|\bbeta'_{J^c}\|_1\big)\\
&=\|\bu\|_1+2\gamma\|\bbeta_{J^c}\|_1+ \gamma\big(\|\bbeta_J\|_1-\|\bbeta'_J\|_1\big) -\gamma\big(\|\bbeta'_{J^c}\|_1+\|\bbeta_{J^c}\|_1\big)\\
&\le \|\bu\|_1+2\gamma\|\bbeta_{J^c}\|_1+\gamma\|\bu_J\|_1-\gamma\|\bu_{J^c}\|_1\\
&=2\gamma\|\bbeta_{J^c}\|_1+(\gamma+1)\|\bu_J\|_1-(\gamma-1)\|\bu_{J^c}\|_1\\
&=2\gamma\|\bbeta_{J^c}\|_1+(\gamma+1)\big(\|\bu_J\|_1-c_\gamma^{-1}\|\bu_{J^c}\|_1\big).
\label{thtranseq3}
\end{align}
If $c_\gamma\|\bu_J\|_1<\|\bu_{J^c}\|_1$, the claim of the lemma is straightforward.
Otherwise, $\|\bu_{J^c}\|_1\le c_\gamma\|\bu_J\|_1$ and using the definition of the compatibility constant we get
\begin{align}
    \frac{2\lambda(\gamma+1)}{\gamma}\big(\|\bu_J\|_1-c_\gamma^{-1}\|\bu_{J^c}\|_1\big)-\|\bfA\bu\|_2^2
    &\le\frac{2\lambda(\gamma+1)}{\gamma}\bigg(\frac{|J|\cdot \|\bfA\bu\|_2^2}{\kappa_{\bfA^2}(J,c_\gamma)}\bigg)^{1/2}-\|\bfA\bu\|_2^2
    \\
    &\le\frac{(\gamma+1)^2\lambda^2|J|}{\gamma^2\kappa_{\bfA^2}(J,c_\gamma)},\qquad[\text{by Cauchy-Schwarz}]
\end{align}
which completes the proof.
\end{proof}

To close this subsection of auxiliary results, we provide simple upper bounds on the quantiles of some
random noise variables.
\begin{proposition}
\label{proptrans:3}
Let $m=N-n$ and $n_{\star}= n\wedge m$. Introduce the random vectors $\bzeta^{(1)} = \frac1{n}\sum_{i=1}^{n} Y_i\bX_i-\esp [Y \bX]$,
\begin{align}
\bzeta &=  \frac1{n}\sum_{i=1}^{n} Y_i\bX_i-\frac1{m}\sum_{i=n+1}^{n+m} f^{\star}(\bX_i)\bX_i\quad\text{and}\quad
\bar\bzeta =  \frac1{n}\sum_{i=1}^{n} Y_i\bX_i-\frac1{N}\sum_{i=1}^{N} f^{\star}(\bX_i)\bX_i.
\end{align}
Under Assumptions \ref{A1} and \ref{A2}, and for any $\delta\in (0,1)$, each of the following inequalities
\begin{align}
\|\bzeta^{(1)}\|_\infty &\le 2B_Y\Big(\frac{\log(2p/\delta)}{n}\Big)^{1/2}\bigg[1+\frac{B_X}{3}\Big(\frac{\log(2p/\delta)}{n}\Big)^{1/2}\bigg]\\
\|\bzeta\|_\infty &\le 2B_Y\Big(\frac{\log(2p/\delta)}{n_{\star}}\Big)^{1/2}\bigg[1+\frac{B_X}{3}\Big(\frac{\log(2p/\delta)}{n_{\star}}\Big)^{1/2}\bigg]\\
\|\bar\bzeta\|_\infty &\le 2B_Y\Big(\frac{\log(2p/\delta)}{n}\Big)^{1/2}\bigg[1+\frac{B_X}{2}\Big(\frac{\log(2p/\delta)}{n}\Big)^{1/2}\bigg]
\end{align}
holds with probability at least $1-\delta$.
\end{proposition}

\begin{proof} We will only prove the inequality corresponding to $\bzeta$. The others being
very similar are left to the reader. Denote $\bmu=\esp[Y\bX]=\esp[f^{\star}(\bX)\bX]\in\R^p$,
and introduce the random vectors
$$
\bZ_{i}=
\begin{cases}
N(Y_i\bX_{i}-\bmu)/n,  & i\in[n],\\
N(\bmu-f^{\star}(\bX_i)\bX_{i})/m,  & i\in [N]\setminus[n].
\end{cases}
$$
The vectors $\bZ_i$ are independent, centered, bounded and satisfy
$$\bzeta=\frac{\bZ_1+\dots+\bZ_N}{N}.$$
Furthermore, Assumption \ref{A2} implies that $\|\bZ_i\|_\infty\le 2N B_YB_X/n$ if
$i\le n$ and that $\|\bZ_i\|_\infty\le 2NB_YB_X/m$ if $i>n$. One can also bound from above the
variance of the $j$-th component $Z_{ij}$ of $\bZ_i$ as follows. If $i\le n$ then, in view of
Assumptions~\ref{A1} and \ref{A2}, $\esp[Z_{ij}^2]\le (N/n)^2 \esp[Y_i^2X_{ij}^2]\le (NB_Y/n)^2$.
Similarly, if $i> n$ then $\esp[Z_{ij}^2]\le (NB_Y/m)^2$. Hence, we may easily deduce that, for
all $j\in[p]$,
$$
\frac{1}{N}\sum_{i=1}^{N}\esp[Z^{2}_{ij}]\le \frac{2NB^{2}_{Y}}{n_{\star}}.
$$
Therefore, using the Bernstein inequality recalled in Proposition \ref{prop:Bernstein_real} of
Appendix \ref{A}, for every $j\in[p]$ and every $\delta>0$, we get that inequality
\begin{align}
|\zeta_j|>2B_Y\left(\frac{\log(2p/\delta)}{n_{\star}}\right)^{1/2}+\frac{2B_YB_X\log(2p/\delta)}{3n_{\star}}
\end{align}
holds with probability at most $\delta/p$. The claim of Proposition \ref{proptrans:3} follows from
the union bound.
\end{proof}
\begin{remark}
One can easily check that the inequality $\esp[Z_{ij}^2]\le (NB_Y/n)^2$, for $i=1,\ldots,n$,
used in the previous proof can be replaced by $\esp[Z_{ij}^2]\le (N L_YB_X/n)^2$, where
$L_Y = (\esp[Y_i^2])^{1/2}$. This may lead to a better risk bound in the cases where the random
variable $Y_i$ is not well concentrated around its average value.
\end{remark}

We are now in a position to prove the main theorems of this paper.

\subsection{Proof of \Cref{thtransductive}}

The proof of Theorem \ref{thtransductive} follows directly from  \Cref{proptrans:3} and \Cref{proptrans:1}  below.
For simplicity, the parameter $\gamma>1$ introduced in Proposition \ref{proptrans:1} is fixed at the value $\gamma=2$ in
\Cref{thtransductive}.
\begin{proposition}
\label{proptrans:1}
Let $\bzeta$ be as in \Cref{proptrans:3}. For any $\gamma>1$, we set $c_\gamma = (\gamma+1)/(\gamma-1)$. On the event
$\mathscr E = \{\|\bzeta\|_\infty\le \lambda/\gamma\}$, for every $\bbeta\in\R^p$ and every $J\subseteq[p]$, we have
\begin{align}
\transRisk(f_{\hat\bbeta}) \le  \transRisk(f_{\bbeta})
    + 4\lambda\|\bbeta_{J^c}\|_1+ \frac{(\gamma+1)^2\lambda^2|J|}{\gamma^2\kappa_{\hat\bfSigma_{\rm unlab}}(J,c_\gamma)}.
    \end{align}
\end{proposition}

\begin{proof}
Along the proof, we will use for convenience the shorthand notations $m=N-n$ and  $\mathbf A=\hat\bfSigma_{\rm unlab}^{1/2}$.
First, notice that developing the square in the expression $\transRisk(f_{\bbeta})=
\frac1{m}\sum_{i=n+1}^N \big(\bX_i^\top\bbeta-f^{\star}(\bX_i)\big)^2$, we get
\begin{align}
\transRisk(f_{\bbeta})
		&= \|\mathbf A\bbeta\|_2^2 - \bigg(\frac2{m}\sum_{i=n+1}^{n+m} f^{\star}(\bX_i)\bX_i^\top\bigg)\bbeta+\frac1{m}\sum_{i=n+1}^{n+m} f^{\star}(\bX_i)^2\\
    &= \|\mathbf A\bbeta\|_2^2 +2\bzeta^\top\bbeta-\frac2n\bY^\top\Xlab\bbeta+\frac1{m}\sum_{i=n+1}^{n+m} f^{\star}(\bX_i)^2.
\end{align}
This implies that for every $\bbeta\in\R^p$, we have
\begin{align}
\transRisk(f_{\hat\bbeta}) - \transRisk(f_{\bbeta})
    &= \|\bfA\hat\bbeta\|_2^2-\|\bfA\bbeta\|_2^2+2\bzeta^\top(\hat\bbeta-\bbeta)-\frac2n\bY^\top\Xlab(\hat\bbeta-\bbeta).
\end{align}
Using \Cref{lem:1} with the convex penalty term $\pen(\bbeta)=2\lambda\Vert \bbeta\Vert_1$ , we deduce that, for every $\bbeta\in\R^p$,
\begin{align}
\transRisk(f_{\hat\bbeta}) - \transRisk(f_{\bbeta})
    &\le 2\bzeta^\top(\bbeta-\hat\bbeta)+2\lambda(\|\bbeta\|_1-\|\hat\bbeta\|_1)-\|\bfA(\bbeta-\hat\bbeta)\|_2^2.\label{thtranseq1}
\end{align}
On the event $\mathscr E$, note that $2\bzeta^\top(\bbeta-\hat\bbeta)\le 2\Vert\bzeta\Vert_{\infty}\Vert\bbeta-\hat\bbeta\Vert_1\le
\frac{2\lambda}{\gamma}\Vert\bbeta-\hat\bbeta\Vert_1$,
which leads to
\begin{align}
2\bzeta^\top(\bbeta-\hat\bbeta)+2\lambda(\|\bbeta\|_1-\|\hat\bbeta\|_1)
&\le \frac{2\lambda}{\gamma}\left(\|\bbeta-\hat\bbeta\|_1+\gamma\|\bbeta\|_1-\gamma\|\hat\bbeta\|_1\right).
\label{thtranseq2}
\end{align}
Combining equations \eqref{thtranseq1} and \eqref{thtranseq2}, we get that
on the event $\mathscr E$, for every $\bbeta\in\R^p$ and every $J\subseteq[p]$,
\begin{equation}
    \transRisk(f_{\hat\bbeta}) - \transRisk(f_{\bbeta})
    \le 2\lambda\gamma^{-1}\big(\|\bbeta-\hat\bbeta\|_1+\gamma\|\bbeta\|_1-\gamma\|\hat\bbeta\|_1\big)-\|\bfA(\bbeta-\hat\bbeta)\|_2^2.
		\label{thtra}
\end{equation}
The claim of the proposition follows from \cref{thtra} by  applying \Cref{lem:2} with $\mu=\lambda$.
\end{proof}
To conclude the proof of \Cref{thtransductive}, it suffices to note that in view of \Cref{proptrans:3}, the probability of
the event $\mathscr E = \{\|\bzeta\|_\infty\le \lambda/\gamma\}$ is larger than $1-\delta$ provided that
$$
\lambda\ge 2\gamma B_Y\Big(\frac{\log(2p/\delta)}{n_{\star}}\Big)^{1/2}\bigg[1+\frac{B_X}{3}\Big(\frac{\log(2p/\delta)}{n_{\star}}\Big)^{1/2}\bigg].
$$

\subsection{Proofs for the semi-supervised version of the lasso}

We start this section by some arguments that are shared by the proofs of both theorems stated in \Cref{sec:4}.
Let $J\subseteq [p]$ and let $\bbeta$ be a minimizer of the right hand side of \eqref{eq:soi-semisup}. Note in
particular that $\bbeta$ is a deterministic vector depending on the unknown distribution $P$ of the data. In addition,
if the model is well-specified and $J=J^\star$ then $\bbeta=\bbeta^\star$. We will also use the notation
$\bu=\hat\bbeta-\bbeta$ and
\begin{equation}
\bzeta^{(1)} = \frac1{n}\sum_{i=1}^{n} Y_i\bX_i-\esp [Y \bX]\quad\mbox{and}\quad\bzeta^{(2)} =\big(\bfSigma-\hat\bfSigma_{\rm all}\big)\bbeta.
\label{eq:def-bnu}
\end{equation}
Furthermore, to ease notation, we set $\hat\bfSigma_N = \hat\bfSigma_{\rm all}$, $\hat\bfSigma_n = \hat\bfSigma_{\rm lab}$,
$\bfA = \hat\bfSigma_N^{1/2}$. First, observe that the excess risk $\mathcal E(f_{\hat\bbeta})
= \int_{\mathcal X}\big(\bx^{\top}\hat\bbeta-f^{\star}(\bx)\big)^2P_X({\rm d}\bx)$ of the predictor $f_{\hat\bbeta}$ satisfies
\begin{align}
    \mathcal E(f_{\hat\bbeta})
                                &= \int_{\mathcal X}\big\{(\bx^{\top}\bu)^2 + 2\bu^\top \bx\big(\bx^\top\bbeta-f^{\star}(\bx)\big)+\big(\bx^\top\bbeta-f^{\star}(\bx)\big)^2\big\}P_X({\rm d}\bx)  \\
                                &=\Vert\bfSigma^{1/2}\bu\Vert^{2}_{2}+2\bu^{\top}\bfSigma\bbeta-2\bu^{\top}\esp\left[\bX f^{\star}(\bX)\right]+\mathcal E(f_{\bbeta}).
    \label{thssms:eq1}
\end{align}
Next, notice that
\begin{equation}
\Vert\bfSigma^{1/2}\bu\Vert^{2}_{2}=\bu^{\top}(\bfSigma-\hat\bfSigma_N)\bu+\Vert\bf A\bu\Vert^{2}_{2},
\label{thssms:eq2}
\end{equation}
and that
\begin{align}
    2\bu^{\top}\bfSigma\bbeta&=2\bu^{\top}(\bfSigma-\hat\bfSigma_N)\bbeta+2\bu^{\top}\hat\bfSigma_N\bbeta \\
                             &=2\bu^{\top}(\bfSigma-\hat\bfSigma_N)\bbeta+\Vert\bf A\hat\bbeta\Vert^{2}_{2}-\Vert\bf A\bu\Vert^{2}_{2}-\Vert\bf A\bbeta\Vert^{2}_{2},
\label{thssms:eq3}
\end{align}
where in the last line we have used the identity $2a^{\top}b=\Vert a+b\Vert^{2}_{2}-\Vert a\Vert^{2}_{2}-\Vert b\Vert^{2}_{2}$ with $a=\bf A\bu$ and $b=\bf A\bbeta$. Transforming \cref{thssms:eq1} thanks to \eqref{thssms:eq2} and \eqref{thssms:eq3} we obtain
\begin{align}
    \mathcal E(f_{\hat\bbeta})-\mathcal E(f_{\bbeta})&=\bu^{\top}(\bfSigma-\hat\bfSigma_N)\bu+2\bu^{\top}(\bfSigma-\hat\bfSigma_N)\bbeta
                                                     + \Vert{\bf A}\hat\bbeta\Vert^{2}_{2}-\Vert{\bf A}\bbeta\Vert^{2}_{2}-2\bu^{\top}\esp\left[Y\bX\right] \\
                                                     &=\bu^{\top}(\bfSigma-\hat\bfSigma_N)\bu+2\bu^{\top}\bzeta^{(2)}
                                                     + \Vert{\bf A}\hat\bbeta\Vert^{2}_{2}-\Vert{\bf A}\bbeta\Vert^{2}_{2}+2\bu^{\top}\bzeta^{(1)}-\frac{2}{n}\bY^{\top}{\bf X}_n\bu,
    \label{thssms:eq4}
\end{align}
where we have used the identity $\esp\left[Y\bX\right]=\esp\left[\bX f^{\star}(\bX)\right]$ and the definitions of $\bzeta^{(1)}$ and $\bzeta^{(2)}$.
Applying Lemma~\ref{lem:1} with $\pen(\bbeta)=2\lambda\Vert \bbeta\Vert_{1}$ and combining its result with \eqref{thssms:eq4}, we arrive at
\begin{equation}
    \mathcal E(f_{\hat\bbeta})-\mathcal E(f_{\bbeta}) \le \underbrace{2\bu\!^{\top}(\bzeta^{(1)}+\bzeta^{(2)}) + 2\lambda(\|\bbeta\|_1-\|\hat\bbeta\|_1)}_{{\bf T}_1}
                                                      + \underbrace{\bu\!^{\top}(\bfSigma-\hat\bfSigma_N)\bu-\Vert{\bf A}\bu\Vert^{2}_{2}}_{{\bf T}_2}.\label{thssms:eq4b}
\end{equation}

\subsubsection{Proof of \Cref{thssws}.}

As mentioned earlier, in the well-specified setting we have $\bbeta=\bbeta^\star$ and, therefore, $\excessRisk(f_{\hat\bbeta}) =
\|\bfSigma^{1/2}\bu\|_2^2$ and $\excessRisk(f_{\bbeta^\star})=0$.  Hence, \eqref{thssms:eq4b} yields
\begin{equation}\label{thssws:eq4}
2\|\hat\bfSigma_{N}^{1/2}\bu\|_2^2\le
	2\bu^\top\big(\bzeta^{(1)}+\bzeta^{(2)}\big)+2\lambda(\|\bbeta^\star\|_1-\|\bbeta^\star+\bu\|_1).
\end{equation}
Combining the duality inequality $|\bu^\top\big(\bzeta^{(1)}+\bzeta^{(2)}\big)|\le \|\bzeta^{(1)}+\bzeta^{(2)}\|_\infty\|\bu\|_1$
with the following one
$\|\bbeta^\star\|_1-\|\bbeta^\star+\bu\|_1 = \|\bbeta^\star_{J^\star}\|_1-\|\bbeta^\star_{J^\star}+\bu_{J^\star}\|_1-\|\bu_{(J^\star)^c}\|_1
\le \|\bu_{J^\star}\|_1-\|\bu_{(J^\star)^c}\|_1$,
we infer from inequality \eqref{thssws:eq4} that on the event
$\mathscr E = \big\{2\|\bzeta^{(1)}+\bzeta^{(2)}\|_\infty\le \lambda\big\}$, we have
\begin{equation}\label{eq:th1:2}
2\|\hat\bfSigma_{N}^{1/2}\bu\|_2^2\le \lambda(3\|\bu_{J^\star}\|_1-\|\bu_{(J^\star)^c}\|_1).
\end{equation}
This implies that $\|\bu_{(J^\star)^c}\|_1\le 3\|\bu_{J^\star}\|_1$ and, therefore,
\begin{equation}
2\bar\kappa_{\hat\bfSigma_{N}}(J^\star,3)\|\bu_{J^\star}\|_1^2\le 2s^\star\|\hat\bfSigma_{N}^{1/2}\bu\|_2^2\le
3\lambda s^\star\|\bu_{J^\star}\|_1.
\end{equation}
This yields $\|\bu_{J^\star}\|_1\le 3\lambda s^\star/(2\bar\kappa_{\hat\bfSigma_{N}}(J^\star,3))$ and, since
$\max_{j,j'}|\bfSigma_{j,j'}|\le 1$, $\displaystyle\|\bfSigma^{1/2}\bu\|_2 \le \|\bu\|_1\le 4\|\bu_{J^\star}\|_1$,
which implies that
\begin{equation}\label{eq:th1:3}
\excessRisk(f_{\hat\bbeta})=\|\bfSigma^{1/2}\bu\|_2^2
		\le {\bigg(\frac{6\lambda s^\star}{\bar\kappa_{\hat\bfSigma_{N}}(J^\star,3)}\bigg)\!}^2 .
\end{equation}
On the other hand, if we denote by $I$ the set of the $s^\star$ largest entries of the vector $|\bu|$, inequality \eqref{eq:th1:2}
implies that $2\|\hat\bfSigma_{N}^{1/2}\bu\|_2^2\le \lambda(3\|\bu_{I}\|_1-\|\bu_{I^c}\|_1)$.

Therefore,  using the definition of the restricted eigenvalue and similar arguments as above, we deduce that $\|\bu_{I}\|_2\le 3\lambda {\color{red}\sqrt{s^\star}}/(2\kappa^{\rm RE}_{\hat\bfSigma_{N}}(I,3))$. Furthermore,
$\|\bu\|_2^2= \|\bu_I\|_2^2+\|\bu_{I^c}\|_2^2\le \|\bu_I\|_2^2+\|\bu_{I^c}\|_\infty\|\bu_{I^c}\|_1\le
\|\bu_I\|_2^2+(s^{\star})^{-1}\|\bu_{I}\|_1\|\bu_{I^c}\|_1\le \|\bu_I\|_2^2+3(s^{\star})^{-1}\|\bu_{I}\|_1^2\le 4\|\bu_I\|_2^2$. This yields
\begin{equation}\label{eq:th1:4}
\excessRisk(f_{\hat\bbeta})=\|\bfSigma^{1/2}\bu\|_2^2 \le \|\bfSigma\|\cdot\|\bu\|_2^2
\le 4\|\bfSigma\|\cdot\|\bu_I\|_2^2\le \frac{9\|\bfSigma\|\lambda^2 s^\star}{\kappa^{\rm RE}_{\hat\bfSigma_{N}}(I,3)^2}\ .
\end{equation}
Combining \eqref{eq:th1:3} and\eqref{eq:th1:4}, we get the first claim of the theorem.

To get the second claim of the theorem, we go back to \eqref{eq:th1:2} and use the following inequalities:
\begin{align}
    2\|\bfSigma^{1/2}\bu\|_2^2
    &= 2\|\hat\bfSigma^{1/2}_{N}\bu\|_2^2+2\bu^\top(\bfSigma-\hat\bfSigma_{N})\bu \\
    &\le  3\lambda\|\bu_{J^\star}\|_1+2{\|\bfSigma-\hat\bfSigma_{N}\|}_{\infty}{\|\bu\|}_{1}^2 \\
    &\le  3\lambda\|\bu_{J^\star}\|_1+32{\|\bfSigma-\hat\bfSigma_{N}\|}_{\infty}{\|\bu_{J^{\star}}\|}_{1}^2.
    \label{thssws:eq5}
\end{align}
In the sequel, let us denote $\kappa=\bar\kappa_{\bfSigma}(J^{\star},3)$ for brevity. Then, upper bounding the two instances of $\|\bu_{J^{\star}}\|_1$ in \eqref{thssws:eq5} by $(s^{\star}\|\bfSigma^{1/2}\bu\|^2_2/\kappa)^{1/2}$, we
infer that on $\mathscr E$,
\begin{align}\label{eq:th1:4}
	\|\bfSigma^{1/2}\bu\|_2^2
			&\le  \frac{3\lambda \sqrt{s^{\star}}}{2\sqrt{\kappa}}\,\|\bfSigma^{1/2}\bu\|_2+
			\frac{16s^{\star}}{\kappa}\,{\|\bfSigma-\hat\bfSigma_{N}\|}_{\infty}\|\bfSigma^{1/2}\bu\|_2^2.
\end{align}
Dividing both sides by $\|\bfSigma^{1/2}\bu\|_2$ (if this quantity vanishes then the claim of the theorem
is obviously true) and after some algebra, we get the inequality
\begin{align}\label{eq:th1:5}
	\|\bfSigma^{1/2}\bu\|_2^2
			&\le  \frac{9\lambda^2 s^{\star}\kappa}{4(\kappa-
			  16s^{\star}\,{\|\bfSigma-\hat\bfSigma_{N}\|}_{\infty})^2}\le \frac{9\lambda^2 s^{\star}}{\kappa},
\end{align}
where the last inequality holds on the event $\mathscr E\cap \{32s^{\star}\,{\|\bfSigma-\hat\bfSigma_{N}\|}_{\infty}\le\kappa\}$.
In view of the union bound, Hoeffding's inequality and Assumption \ref{A2}, we get for any $t>0$,
\begin{align}\label{eq:th1:6}
\prob\left({\|\bfSigma-\hat\bfSigma_{N}\|}_{\infty}\ge t\right)
	&\le p^2 \max_{j,j'\in[p]} \prob\left(|\sigma_{jj'}-\hat\sigma_{jj'}|\ge t\right) \le 2p^2 \exp\left(-{2Nt^2}/{B_X^4}\right),
\end{align}
where $\bfSigma=(\sigma_{ij})$ and $\hat\bfSigma_{N}=(\hat\sigma_{ij})$. Therefore, if
$$
16s^{\star}B_X^2\Big(\frac{2\log(4p^2/\delta)}{N}\Big)^{1/2}\le \kappa,
$$
then the event $\{32s^{\star}\,{\|\bfSigma-\hat\bfSigma_{N}\|}_{\infty}\le\kappa\}$ has a probability
larger than $1-(\delta/2)$. To bound the probability of $\mathscr E$, we use the fact that $\bzeta^{(1)}+\bzeta^{(2)}=
\bar\bzeta$ and the quantiles of the supremum norm of the random vector $\bar\bzeta$ have been assessed
in \Cref{proptrans:3}. This implies that the choice
$$
\lambda\ge 4B_Y\Big(\frac{\log(4p/\delta)}{n}\Big)^{1/2}+\frac{B_XB_Y\log(4p/\delta)}{n}
$$
guarantees that $P(\mathscr E) = P(\|\bzeta\|_\infty\le \lambda/2)\ge 1-(\delta/2)$. This completes the proof.


\subsubsection{Proof of \Cref{thssms}.}

We start by some auxiliary results before providing the proof of the theorem.

\begin{proposition}
\label{thssms:prop1}
Let $J\subseteq [p]$ and let $\bbeta$ be a minimizer of the right hand side of \eqref{eq:soi-semisup}. On the event
$\mathscr E= \mathscr E_1 \cap \mathscr E_2 \cap \mathscr E_3$, where
\begin{align}
\mathscr E_1 =
\big\{
\|\bzeta^{(1)}\|_\infty\le {\textstyle\frac \lambda 4}
\big\},
\quad
\mathscr E_2 =
\big\{
\|\bzeta^{(2)}\|_\infty\le {\textstyle\frac \lambda 4}
\big\},
\quad\mbox{and}\quad
\mathscr E_3 =
\big\{
    \lambda_{\min}(\bfSigma^{-1/2}\hat\bfSigma_N\bfSigma^{-1/2})\ge {\textstyle\frac 23}
\big\},
\qquad
\end{align}
we have
$$\excessRisk (f_{\hat\bbeta}) -\excessRisk (f_{\bbeta})\le4\lambda\|\bbeta_{J^c}\|_1+ \frac{9\lambda^2|J|}{2\kappa_{\hat\bfSigma_{\rm all}}(J,3)}.$$
\end{proposition}

\begin{proof}
Our starting point in this proof is \eqref{thssms:eq4}.
We first focus on bounding ${\bf T}_{1}$.
On the event $\mathscr E_1 \cap \mathscr E_2$, we have
\begin{equation}
    {\bf T}_{1}\le 2\Vert\bzeta^{(1)}+\bzeta^{(2)}\Vert_{\infty}\Vert \bu\Vert_{1}+2\lambda(\|\bbeta\|_1-\|\hat\bbeta\|_1)
    \le \lambda\big(\Vert \bu\Vert_{1}+2\|\bbeta\|_1-2\|\hat\bbeta\|_1\big).
\label{thssms:eq5}
\end{equation}
We now look for an upper bound of the term ${\bf T}_{2}$. On the event $\mathscr E_3$, for any $\bv \in\R^p$,
\begin{align}
    \bv^\top\big(2\bfI_p -3\bfSigma^{-1/2}\hat\bfSigma_N\bfSigma^{-1/2}\big)\bv \le 0,
\end{align}
which leads to
\begin{equation}
    \bv^\top\big(\bfI_p - 2\bfSigma^{-1/2}\hat\bfSigma_N\bfSigma^{-1/2}\big)\bv\le -{\textstyle \frac 1 2 }(\bv^\top\bfSigma^{-1/2}\hat\bfSigma_N\bfSigma^{-1/2})\bv.
    \label{thssms:eq6}
\end{equation}
Therefore, applying \eqref{thssms:eq6} to $\bv = \bfSigma^{1/2} \bu$, it follows that on the event $\mathscr E_3$
\begin{align}
    {\bf T}_{2}
    =  \bv^\top \big(\bfI_p - 2\bfSigma^{-1/2} \hat\bfSigma_N\bfSigma^{-1/2}\big)\bv
    \le   -{\textstyle \frac 1 2 }    \bv^\top(\bfSigma^{-1/2} \hat\bfSigma_N\bfSigma^{-1/2})\bv
    = - {\textstyle \frac 1 2 } \|\bfA\bu\|_2^2.
\label{thssms:eq7}
\end{align}
To sum up, equations \eqref{thssms:eq5} and \eqref{thssms:eq7} together imply that on the event $\mathscr E=\mathscr E_{1}\cap\mathscr E_{2}\cap\mathscr E_{3}$,
\begin{align}
\mathcal E(f_{\hat\bbeta}) - \mathcal E(f_{\bbeta})
    &\le
    \lambda\big(\Vert \bu\Vert_{1}+2\|\bbeta\|_1-2\|\hat\bbeta\|_1\big)-{\textstyle \frac 1 2 }\|\bfA\bu\|_2^2.
\end{align}
The desired result follows from this inequality and \Cref{lem:2} with $\mu=\lambda$ and $\gamma=2$.
\end{proof}

Note that according to \Cref{proptrans:3},
\begin{equation}\label{eqtrans:7}
  \prob\bigg(\Vert \bzeta^{(1)} \Vert_\infty\le
  2B_Y\Big(\frac{\log(6p/\delta)}{n}\Big)^{1/2}\bigg[1+\frac{B_X}{3}\Big(\frac{\log(6p/\delta)}{n}\Big)^{1/2}\bigg]\bigg)\ge 1-\frac{\delta}3.
\end{equation}
The next two lemmas provide bounds for the probabilities of the events $\mathscr E_{2}$ and $\mathscr E_{3}$
introduced in \Cref{thssms:prop1}.

\begin{lemma}\label{lem:4}
    Let assumption \ref{A2} be fulfilled.
    Let $J\subseteq[p]$ and let $\bbeta$ be a minimizer
    of the right hand side of \eqref{eq:soi-semisup}. Then, for all $\delta\in(0,1)$, the inequality
    \begin{align}
            \Vert \bzeta^{(2)} \Vert_\infty
            \ge
            B_X B_Y
                \Big(\frac{2\log(6 p/\delta)}{N}\Big)^{1/2}\Big[
                1+
                \frac{B_X}{3} \Big(\frac{2p\Vert \bfSigma^{-1} \Vert \log(6 p/\delta)}{N}\Big)^{1/2}\Big]
    \end{align}		
holds with probability at most $\delta/3$, where the random vector $\bzeta^{(2)}$ is defined in \cref{eq:def-bnu}.
\end{lemma}
\begin{proof}
    Note that $\bzeta^{(2)}=(1/N)\sum_{i=1}^{N}\bs U_{i}$, where $\bs U_{i}=\bX_{i}(\bX^{\top}_{i}\bbeta)-\esp[\bX(\bX^{\top}\bbeta)]$.
		The random vectors $\bs U_{i}$ are independent and, for all $i\in[N]$ and all $j\in[p]$, the $j$-th component $U_{ij}=X_{ij}(\bX^{\top}_{i}\bbeta)-\esp[X_{j}(\bX^{\top}\bbeta)]$ of $\bs U_{i}$ satisfies, almost surely,
    \begin{equation}
        | U_{ij} |
        \le 2B_X^2 \Vert \bbeta \Vert_1
        \le 2B_X^2  \sqrt p \Vert \bbeta \Vert_2,
   \end{equation}
     where we have used that $\vert\bX^{\top}\bbeta\vert\le\Vert\bX\Vert_{\infty} \Vert\bbeta\Vert_{1}\le B_{X}\Vert\bbeta\Vert_{1}$ with probability $1$. Then, noticing that
     $\Vert\bbeta\Vert_{2}=\Vert\bfSigma^{-1/2}\bfSigma^{1/2}\bbeta\Vert_{2}\le\Vert\bfSigma^{-1/2}\Vert\Vert\bfSigma^{1/2}\bbeta\Vert_{2}=\Vert\bfSigma^{-1}\Vert^{1/2}\Vert\bfSigma^{1/2}\bbeta\Vert_{2}$, we deduce that
     \begin{equation}
\vert U_{ij}\vert\le 2B_X^2  (p \Vert \bfSigma^{-1} \Vert)^{1/2} \Vert \bfSigma^{1/2} \bbeta \Vert_2,
    \end{equation}
    almost surely.
    Since $\bbeta$ minimizes the term on the right hand side of \eqref{eq:soi-semisup}, by \Cref{lem:strong-convexity-minimizer} below,
    $\Vert \bfSigma^{1/2} \bbeta \Vert_2\le B_{Y}$.
    Thus for all $i\in[N]$ and all $j\in[p]$,
    $\vert U_{ij}\vert\le 2B^{2}_{X}B_{Y}(p \Vert \bfSigma^{-1} \Vert)^{1/2}$.
Furthermore, according to the previous lines, it holds
$\frac1N\sum_{i=1}^{N}\esp[X^{2}_{ij}(\bX^{\top}_{i}\bbeta)^{2}]\le B^{2}_{X}B^{2}_{Y}$.
\Cref{prop:Bernstein_real} and the union bound complete the proof.
\end{proof}

\begin{lemma}
    \label{lem:lambda-min}
    Under assumption \ref{A2}, the smallest eigenvalue $\lambda_{\min}(\bfSigma^{-1/2}\hat\bfSigma_N\bfSigma^{-1/2})$ of the
    matrix $\bfSigma^{-1/2}\hat\bfSigma_N\bfSigma^{-1/2}$ satisfies
    \begin{equation}
        \label{eq:concentration-lambdamin}
        \prob
        \bigg\{
            \lambda_{\rm min}(\bfSigma^{-1/2}\hat\bfSigma_N\bfSigma^{-1/2})
            \ge 1 - \Big(\frac{2 B_X^2 p \|\bfSigma^{-1}\| \log(p/\delta)}{N}\Big)^{1/2}
            \bigg\}\ge 1-\delta,
    \end{equation}
    for all $\delta\in(0,1)$ such that
    $2 B_X^2 p \|\bfSigma^{-1}\| \log(p/\delta)\le N$.
\end{lemma}
\begin{proof}
    For all $i\in[N]$, $\lambda_{\rm max} (\bfSigma^{-1/2}\bX_i \bX_i^\top\bfSigma^{-1/2}) = \Vert \bfSigma^{-1/2}\bX_i\Vert^2 \le p B_X^2\|\bfSigma^{-1}\|$
    and the matrix $\bfSigma^{-1/2}\bX_i \bX_i^\top\bfSigma^{-1/2}$ is positive semi-definite. Applying the first Chernoff matrix inequality given in Remark 5.3 of \citet{T12}
    to the sequence of matrices  $\{\bfSigma^{-1/2}\bX_i \bX_i^\top\bfSigma^{-1/2}:i\in[N]\}$ with
    \begin{equation}
                t = 1 - \Big(
												\frac{2 B_X^2 p\|\bfSigma^{-1}\| \log(p/\delta)}{ N}	
												\Big)^{1/2},
                \qquad
                R = p B_X^2,
                \qquad
                \delta =
                p
            \exp\Big\{
                - \frac{(1-t)^2 N}{2R\|\bfSigma^{-1}\|}
            \Big\}
    \end{equation}
    yields \eqref{eq:concentration-lambdamin}.
\end{proof}

\begin{lemma}
    \label{lem:strong-convexity-minimizer}
    Let $\pen:\R^p\rightarrow [0,+\infty)$ be a convex function such that $\pen(\bzero_p) = 0$.
    Let $\bar\bbeta$ be a minimizer of the function
    \begin{equation}
    \Phi(\bbeta) =  \esp[(\bbeta^\top \bX - Y)^2] + \pen(\bbeta), \qquad \bbeta\in\R^p.
    \end{equation}
    Then $\esp[(\bar\bbeta^\top \bX)^2] \le \esp[Y^2]$ and,
    if Assumption \ref{A2} is fulfilled, $\esp[(\bar\bbeta^\top \bX)^2] \le B_Y^2$.
\end{lemma}
\begin{proof}
    We apply \Cref{lem:1} with $\bfA = \esp[\bX \bX^\top]^{1/2}$, $n=1$,
    $\bY = 1$ and $\Xlab = \esp[Y\bX]$ so that
    $\frac 1 n \bY^\top \Xlab = \esp [Y\bX]$.
    Inequality \eqref{eq:1} with $\bbeta=\bzero_p$ yields
    \begin{equation}
        \esp[(\bar\bbeta^\top \bX)^2]
        \le
        2 \esp[ Y ( \bar\bbeta^\top\bX) ]
        - \pen(\bar\bbeta)
        - \esp[(\bar\bbeta^\top \bX)^2].
    \end{equation}
    Rearranging the terms and using that $\pen(\bar\bbeta)\ge0$, we get
    $\esp[(\bar\bbeta^\top \bX)^2] \le \esp[ Y ( \bar\bbeta^\top\bX) ]$.
    In view of the Cauchy-Schwarz inequality,
    $(\esp[ Y ( \bar\bbeta^\top\bX) ])^2 \le \esp[Y^2]\, \esp[(\bar\bbeta^\top \bX)^2]$, which implies that
    $(\esp[(\bar\bbeta^\top \bX)^2])^2 \le {\esp[Y^2] \,\esp[(\bar\bbeta^\top \bX)^2]}$. It now suffices to divide both sides of the last inequality
    by ${\esp[(\bar\bbeta^\top \bX)^2]}$ to obtain the claim of the lemma.
\end{proof}

\begin{proof}[Proof of \cref{thssms}]
Under the conditions of the theorem, we have
$$
\Big(\frac{ 2B_X^2 p\Vert \bfSigma^{-1}\Vert \log(3p/\delta)}{N}\Big)^{1/2}\le \frac13.
$$
Therefore, \Cref{lem:lambda-min} implies that $\prob(\mathscr E_3)\ge 1-\delta/3$. On the other hand, in view of \cref{eqtrans:7} and \Cref{lem:4},
the conditions
\begin{align}
\lambda&\ge 8B_Y\Big(\frac{\log(6p/\delta)}{n}\Big)^{1/2}\Big[1+\frac{B_X}{3}\Big(\frac{\log(6p/\delta)}{n}\Big)^{1/2}\Big],\\
\lambda&\ge 4B_X B_Y \Big(\frac{2\log(6 p/\delta)}{N}\Big)^{1/2}\Big[1+
                \frac{B_X}{3} \Big(\frac{2p\Vert \bfSigma^{-1} \Vert \log(6 p/\delta)}{N}\Big)^{1/2}\Big]
\end{align}
imply that $\prob(\mathscr E_1)\ge 1-\delta/3$ and $\prob(\mathscr E_2)\ge 1-\delta/3$. One can easily check that under the conditions
of the theorem, the two inequalities of the last display are satisfied. Therefore, we have $\prob(\mathscr E_1\cap \mathscr E_2
\cap \mathscr E_3)\ge 1-\delta$. Finally, applying \Cref{thssms:prop1} we get the claim of the theorem.
\end{proof}

\subsubsection{Proof of the oracle inequality in expectation.}\label{ssec:74}

Let $\delta$ be a positive number smaller than $1$ to be chosen later.
We have already seen in \Cref{cor:1} that on an event $\mathscr E$ of probability $1-\delta$, we have
    \begin{align}
    \excessRisk (f_{\hat\bbeta}) \le
    \inf_{J\subseteq [p]}
    \inf_{\bbeta\in\R^p}
    \Big\{
        \excessRisk (f_{\bbeta})
        + 4\lambda\|\bbeta_{J^c}\|_1
        + {\textstyle\frac{27\| \bfSigma^{-1} \|}{4}}\, \lambda^2 |J|
    \Big\}.
    \end{align}
On the other hand, using the fact that $\hat\bbeta$ minimises the function $\psi(\bbeta) = \|\hat\bfSigma_N^{1/2}\bbeta\|_2^2 -\frac2n\bY^\top\bfX_n\bbeta+2\lambda\|\bbeta\|_1$,
we have $\psi(\hat\bbeta)\le \psi(\mathbf 0_p)$, which yields
$$
\|\hat\bfSigma_N^{1/2}\hat\bbeta\|_2^2 -\frac2n\bY^\top\bfX_n\hat\bbeta+2\lambda\|\hat\bbeta\|_1 =
\|\hat\bfSigma_N^{1/2}\hat\bbeta-{\textstyle\frac1n}\hat\bfSigma_N^{-1/2}\bfX_n^\top\bY\|_2^2 -{\frac1{n^2}}\|\hat\bfSigma_N^{-1/2}\bfX_n^\top\bY\|_2^2+2\lambda\|\hat\bbeta\|_1 \le 0.
$$
Note that $\hat\bfSigma_N^{-1/2}$ is understood as the Moore-Penrose pseudo-inverse and all the expressions involving this quantity
are well defined since $N\hat\bfSigma_N \succeq n\hat\bfSigma_n = \bfX_n^\top\bfX_n$.
This implies that $2\lambda\|\hat\bbeta\|_1\le {\textstyle\frac1{n^2}}\|\hat\bfSigma_N^{-1/2}\bfX_n^\top\bY\|_2^2\le
\frac1{n^2}\|\hat\bfSigma_N^{-1/2}\bfX_n^\top\|^2\|\bY\|_2^2=\frac1{n}\|\hat\bfSigma_N^{-1/2}\hat\bfSigma_n\hat\bfSigma_N^{-1/2}\|
\,\|\bY\|_2^2$, which entails
\begin{equation}
\|\hat\bbeta\|_1\le\frac{B_Y^2}{2\lambda}\|\hat\bfSigma_N^{-1/2}\hat\bfSigma_n\hat\bfSigma_N^{-1/2}\|
\le \frac{B_Y^2N}{2n\lambda}.\label{norm:1}
\end{equation}
It is also true that for every $\bbeta\in\RR^p$,
\begin{align}
    \excessRisk (f_{\bbeta})
        &= \esp[(f^\star(\bX)-\bX^\top\bbeta)^2]\le 2\esp[f^\star(\bX)^2]+2\bbeta^\top\bfSigma\bbeta
        \le 2B_Y^2+2\|\bbeta\|_1^2.\label{norm:2}
\end{align}
Therefore, we have
$\esp[\excessRisk (f_{\hat\bbeta})\mathds 1_{\mathscr E^c}]\le 2B_Y^2\prob(\mathscr E^c) +2\esp[\|\hat\bbeta\|_1^2\mathds 1_{\mathscr E^c}]
= 2\delta B_Y^2+2\esp[\|\hat\bbeta\|_1^2\mathds 1_{\mathscr E^c}]$. Combining this inequality with \eqref{norm:1}, we get
$$
\esp[\excessRisk (f_{\hat\bbeta})\mathds 1_{\mathscr E^c}]\le
2\delta B_Y^2 +\frac{\delta B_Y^4N^2}{2n^2\lambda^2}.
$$
Setting $\delta = N^{-2}$, we get the claim of the theorem.

\appendix

\section{Bernstein inequality}
\label{A}
The next result follows from \cite[Proposition 2.9]{Massart2007}.
\begin{proposition}
\label{prop:Bernstein_real}
Let $Z_1,\ldots,Z_N$ be independent
real-valued
random variables satisfying, for all $i\in[N]$ and  for some constant $b$,  $\esp[Z_i^2] < +\infty$
and $|Z_i - \esp Z_i |\le b$ almost surely. Denote $\bar Z_N = \frac1N\sum_{i=1}^N Z_i$ and
$\sigma_N^2 = (1/N) \sum_{i=1}^N \esp\left[Z_i^2 - (\esp Z_i)^2\right]$.
Then, for all $\delta\in(0,1)$, inequality
\begin{equation}
    | \bar Z_N-\esp[\bar Z_N]|\le \sigma_N \Big(\frac{ 2  \log(2 /\delta)}{N}\Big)^{1/2}\Big[1 + \frac{b}{6N\sigma_N}\Big(\frac{ 2  \log(2 /\delta)}{N}\Big)^{1/2}\Big],
\end{equation}
holds with probability at least $1-\delta$.
\end{proposition}

\begin{proof}
Define, for all $i\in[N]$, the random variable $X_{i}=(Z_{i}-\esp[Z_{i}])/N$.
Denote as well
$$v=\sum_{i=1}^{N}\esp[X^{2}_{i}]=\frac{1}{N^{2}}\sum_{i=1}^{N}\esp\left[Z^{2}_{i}-(\esp Z_{i})^{2}\right]=\frac{u}{N}.$$
For all $k\ge3$, the assumptions imply that
$$\sum_{i=1}^{N}\esp[(X_{i})^{k}_{+}]\le v\left(\frac{b}{N}\right)^{k-2}\le \frac{k!}{2}v\left(\frac{b}{3N}\right)^{k-2},$$
where we have used the fact that $k!/3^{k-2}\ge2$, for all $k\ge3$. As a result, applying \cite[Prop.\ 2.9]{Massart2007},
with $v=\sigma^2_N/N$ and $c=b/3N$, we get that for all $\delta\in(0,1)$, the inequality
$$
\sum_{i=1}^{N}X_{i}> \sigma_N\sqrt{\frac{2\log(2/\delta)}{N}}+\frac{b\log(2/\delta)}{3N}
$$
holds with probability less than $\delta/2$. Applying the same argument to the variables $-X_{i}$,
we infer that for all $\delta\in(0,1)$, the inequality
$$\sum_{i=1}^{N}X_{i}< -\sigma_N\sqrt{\frac{2\log(2/\delta)}{N}}-\frac{b\log(2/\delta)}{3N},$$
holds with probability less than $\delta/2$, which completes the proof.
\end{proof}

\section*{Acknowledgments}
The work of Q.~Paris was supported by the Russian Academic Excellence Project 5-100. The work of A.~Dalalyan and P.~Bellec was partially
supported by the grant Investissements d'Avenir (ANR-11-IDEX-0003/Labex Ecodec/ANR-11-LABX-0047) and the chair
``LCL/GENES/Fondation du risque, Nouveaux enjeux pour nouvelles donn\'ees''.

\bibliography{bibBDGP16}

\begin{thebibliography}{38}
\providecommand{\natexlab}[1]{#1}
\providecommand{\url}[1]{\texttt{#1}}
\expandafter\ifx\csname urlstyle\endcsname\relax
  \providecommand{\doi}[1]{doi: #1}\else
  \providecommand{\doi}{doi: \begingroup \urlstyle{rm}\Url}\fi

\bibitem[Vapnik(1998)]{Vapnik98}
Vladimir~N. Vapnik.
\newblock \emph{Statistical learning theory}.
\newblock Adaptive and Learning Systems for Signal Processing, Communications,
  and Control. John Wiley \& Sons, Inc., New York, 1998.
\newblock A Wiley-Interscience Publication.

\bibitem[Balcan et~al.(2005)Balcan, Blum, Choi, Lafferty, Pantano, Rwebangira,
  and Zhu]{FreeFoodCam05}
Maria-Florina Balcan, Avrim Blum, Patrick~Pakyan Choi, John Lafferty, Brian
  Pantano, Mugizi~R. Rwebangira, and Xiaojin Zhu.
\newblock Person identification in webcam images: An application of
  semi-supervised learning.
\newblock \emph{ICML2005 Workshop on Learning with Partially Classified
  Training Data}, 2005.

\bibitem[Guillaumin et~al.(2010)Guillaumin, Verbeek, and
  Schmid]{GuillauminVS10}
Matthieu Guillaumin, Jakob~J. Verbeek, and Cordelia Schmid.
\newblock Multimodal semi-supervised learning for image classification.
\newblock In \emph{The Twenty-Third {IEEE} Conference on Computer Vision and
  Pattern Recognition, {CVPR} 2010, San Francisco, CA, USA, 13-18 June 2010},
  pages 902--909, 2010.
\newblock URL \url{http://dx.doi.org/10.1109/CVPR.2010.5540120}.

\bibitem[Brouard et~al.(2011)Brouard, d'Alch{\'{e}}{-}Buc, and
  Szafranski]{BrouarddS11}
C{\'{e}}line Brouard, Florence d'Alch{\'{e}}{-}Buc, and Marie Szafranski.
\newblock Semi-supervised penalized output kernel regression for link
  prediction.
\newblock In Lise Getoor and Tobias Scheffer, editors, \emph{Proceedings of the
  28th International Conference on Machine Learning, {ICML} 2011, Bellevue,
  Washington, USA, June 28 - July 2, 2011}, pages 593--600. Omnipress, 2011.

\bibitem[Chapelle et~al.(2006)Chapelle, Sh\"{o}lkopf, and Zien]{CSZ06}
O.~Chapelle, B.~Sh\"{o}lkopf, and A.~Zien, editors.
\newblock \emph{Semi-Supervised Learning}.
\newblock MIT Press, 2006.

\bibitem[Zhu(2008)]{Zhu08}
X.~Zhu.
\newblock Semi-supervised learning literature survey.
\newblock Technical report, University of Wisconsin -- Madison, 2008.

\bibitem[Rigollet(2007)]{Rigollet07}
Philippe Rigollet.
\newblock Generalized error bounds in semi-supervised classification under the
  cluster assumption.
\newblock \emph{J. Mach. Learn. Res.}, 8:\penalty0 1369--1392, 2007.

\bibitem[Wang and Shen(2007)]{Wang07}
Junhui Wang and Xiaotong Shen.
\newblock Large margin semi-supervised learning.
\newblock \emph{J. Mach. Learn. Res.}, 8:\penalty0 1867--1891, 2007.

\bibitem[Lafferty and Wasserman(2007)]{lafferty2007}
John~D. Lafferty and Larry~A. Wasserman.
\newblock Statistical analysis of semi-supervised regression.
\newblock In \emph{NIPS}, pages 801--808. Curran Associates, Inc., 2007.

\bibitem[Belkin et~al.(2006)Belkin, Niyogi, and Sindhwani]{Belkin06}
Mikhail Belkin, Partha Niyogi, and Vikas Sindhwani.
\newblock Manifold regularization: a geometric framework for learning from
  labeled and unlabeled examples.
\newblock \emph{J. Mach. Learn. Res.}, 7:\penalty0 2399--2434, 2006.

\bibitem[Nadler et~al.(2009)Nadler, Srebro, and Zhou]{Nadler2009}
Boaz Nadler, Nathan Srebro, and Xueyuan Zhou.
\newblock Statistical analysis of semi-supervised learning: The limit of
  infinite unlabelled data.
\newblock In \emph{Advances in Neural Information Processing Systems 22}, pages
  1330--1338. Curran Associates, Inc., 2009.

\bibitem[Niyogi(2013)]{Niyogi13}
Partha Niyogi.
\newblock Manifold regularization and semi-supervised learning: Some
  theoretical analyses.
\newblock \emph{Journal of Machine Learning Research}, 14:\penalty0 1229--1250,
  2013.
\newblock URL \url{http://jmlr.org/papers/v14/niyogi13a.html}.

\bibitem[Sun and Shawe-Taylor(2010)]{Shiliang10}
Shiliang Sun and John Shawe-Taylor.
\newblock Sparse semi-supervised learning using conjugate functions.
\newblock \emph{J. Mach. Learn. Res.}, 11:\penalty0 2423--2455, 2010.

\bibitem[Tibshirani(1996)]{Tib96}
Robert Tibshirani.
\newblock Regression shrinkage and selection via the lasso.
\newblock \emph{J. Roy. Statist. Soc. Ser. B}, 58\penalty0 (1):\penalty0
  267--288, 1996.

\bibitem[B{\"u}hlmann and van~de Geer(2011)]{bookBuhlman2011}
Peter B{\"u}hlmann and Sara van~de Geer.
\newblock \emph{Statistics for high-dimensional data}.
\newblock Springer Series in Statistics. Springer, Heidelberg, 2011.
\newblock Methods, theory and applications.

\bibitem[Bellec et~al.(2016)Bellec, Lecu{\'e}, and Tsybakov]{BLT16}
Pierre~C. Bellec, Guillaume Lecu{\'e}, and Alexandre~B. Tsybakov.
\newblock Slope meets lasso: improved oracle bounds and optimality.
\newblock Technical Report 1605.08651, arXiv, June 2016.

\bibitem[Koltchinskii et~al.(2011)Koltchinskii, Lounici, and Tsybakov]{KLT11}
Vladimir Koltchinskii, Karim Lounici, and Alexandre~B. Tsybakov.
\newblock Nuclear-norm penalization and optimal rates for noisy low-rank matrix
  completion.
\newblock \emph{The Annals of Statistics}, 39\penalty0 (5):\penalty0
  2302--2329, 2011.

\bibitem[Sun and Zhang(2012)]{Sun12}
Tingni Sun and Cun-Hui Zhang.
\newblock Scaled sparse linear regression.
\newblock \emph{Biometrika}, 99\penalty0 (4):\penalty0 879--898, 2012.

\bibitem[Dalalyan et~al.(2014)Dalalyan, Heibiri, and Lederer]{DHL14}
Arnak~S. Dalalyan, Mohamed Heibiri, and Johannes Lederer.
\newblock On the prediction performance of the lasso.
\newblock \emph{Bernoulli}, in press, 2014.

\bibitem[Ye and Zhang(2010)]{Fei10}
Fei Ye and Cun-Hui Zhang.
\newblock Rate minimaxity of the {L}asso and {D}antzig selector for the
  {$\ell_q$} loss in {$\ell_r$} balls.
\newblock \emph{J. Mach. Learn. Res.}, 11:\penalty0 3519--3540, 2010.

\bibitem[Raskutti et~al.(2011)Raskutti, Wainwright, and Yu]{Raskutti11}
Garvesh Raskutti, Martin~J. Wainwright, and Bin Yu.
\newblock Minimax rates of estimation for high-dimensional linear regression
  over {$\ell_q$}-balls.
\newblock \emph{IEEE Trans. Inform. Theory}, 57\penalty0 (10):\penalty0
  6976--6994, 2011.

\bibitem[Rigollet and Tsybakov(2011)]{Rigollet11}
Philippe Rigollet and Alexandre Tsybakov.
\newblock Exponential screening and optimal rates of sparse estimation.
\newblock \emph{Ann. Statist.}, 39\penalty0 (2):\penalty0 731--771, 2011.

\bibitem[Rigollet and Tsybakov(2012)]{Rigollet12a}
Philippe Rigollet and Alexandre~B. Tsybakov.
\newblock Sparse estimation by exponential weighting.
\newblock \emph{Statist. Sci.}, 27\penalty0 (4):\penalty0 558--575, 2012.

\bibitem[Bickel et~al.(2009)Bickel, Ritov, and Tsybakov]{BRT}
Peter~J. Bickel, Ya'acov Ritov, and Alexandre~B. Tsybakov.
\newblock Simultaneous analysis of lasso and {D}antzig selector.
\newblock \emph{Ann. Statist.}, 37\penalty0 (4):\penalty0 1705--1732, 2009.

\bibitem[van~de Geer and B{\"u}hlmann(2009)]{VandeGeerConditionLasso09}
Sara van~de Geer and Peter B{\"u}hlmann.
\newblock On the conditions used to prove oracle results for the {L}asso.
\newblock \emph{Electron. J. Stat.}, 3:\penalty0 1360--1392, 2009.

\bibitem[Koltchinskii(2011)]{koltchinskii2011}
Vladimir Koltchinskii.
\newblock \emph{Oracle Inequalities in Empirical Risk Minimization and Sparse
  Recovery Problems: Ecole d'Et{\'e} de Probabilit{\'e}s de Saint-Flour
  XXXVIII-2008}, volume~38.
\newblock Springer, 2011.

\bibitem[Alquier and Hebiri(2012)]{Alquier2012}
Pierre Alquier and Mohamed Hebiri.
\newblock Transductive versions of the {LASSO} and the dantzig selector.
\newblock \emph{Journal of Statistical Planning and Inference}, 142\penalty0
  (9):\penalty0 2485 -- 2500, 2012.

\bibitem[Lecu{\'e} and Mendelson(2016)]{LecueMend16}
Guillaume Lecu{\'e} and Shahar Mendelson.
\newblock Regularization and the small-ball method i: sparse recovery.
\newblock Technical Report 1601.05584, arXiv, January 2016.

\bibitem[Raskutti et~al.(2010)Raskutti, Wainwright, and Yu]{RWY2010}
Garvesh Raskutti, Martin~J Wainwright, and Bin Yu.
\newblock Restricted eigenvalue properties for correlated gaussian designs.
\newblock \emph{The Journal of Machine Learning Research}, 11:\penalty0
  2241--2259, 2010.

\bibitem[Oliveira(2013)]{oliveira2013lower}
Roberto~Imbuzeiro Oliveira.
\newblock The lower tail of random quadratic forms, with applications to
  ordinary least squares and restricted eigenvalue properties.
\newblock \emph{arXiv preprint arXiv:1312.2903}, 2013.

\bibitem[Rudelson and Zhou(2013)]{rudelson2013reconstruction}
Mark Rudelson and Shuheng Zhou.
\newblock Reconstruction from anisotropic random measurements.
\newblock \emph{Information Theory, IEEE Transactions on}, 59\penalty0
  (6):\penalty0 3434--3447, 2013.

\bibitem[Juditsky and Nemirovski(2011)]{juditsky2011accuracy}
Anatoli Juditsky and Arkadi Nemirovski.
\newblock Accuracy guarantees for-recovery.
\newblock \emph{Information Theory, IEEE Transactions on}, 57\penalty0
  (12):\penalty0 7818--7839, 2011.

\bibitem[Belloni et~al.(2014)Belloni, Chernozhukov, and Wang]{BCW2014}
Alexandre Belloni, Victor Chernozhukov, and Lie Wang.
\newblock Pivotal estimation via square-root lasso in nonparametric regression.
\newblock \emph{Ann. Statist.}, 42\penalty0 (2):\penalty0 757--788, 04 2014.
\newblock \doi{10.1214/14-AOS1204}.
\newblock URL \url{http://dx.doi.org/10.1214/14-AOS1204}.

\bibitem[Pensky(2014)]{Pensky14}
M.~Pensky.
\newblock Solution of linear ill-posed problems using overcomplete
  dictionaries.
\newblock Technical Report 1408.3386, Ann. Statist., to appear, arXiv, August
  2014.

\bibitem[{Vershynin}(2010)]{Vershynin11}
R.~{Vershynin}.
\newblock Introduction to the non-asymptotic analysis of random matrices.
\newblock \emph{ArXiv e-prints}, November 2010.

\bibitem[Bah and Tanner(2014)]{bah14}
Bubacarr Bah and Jared Tanner.
\newblock Bounds of restricted isometry constants in extreme asymptotics:
  formulae for {G}aussian matrices.
\newblock \emph{Linear Algebra Appl.}, 441:\penalty0 88--109, 2014.

\bibitem[Tropp(2012)]{T12}
Joel~A. Tropp.
\newblock User-friendly tail bounds for sums of random matrices.
\newblock \emph{Foundations of Computational Mathematics}, 12\penalty0
  (4):\penalty0 389--434, 2012.

\bibitem[Massart(2007)]{Massart2007}
Pascal Massart.
\newblock \emph{Concentration Inequalities and Model Selection: Ecole d'Et{\'e}
  de Probabilit{\'e}s de Saint-Flour XXXIII - 2003}, volume 1896.
\newblock Springer, 2007.

\end{thebibliography}

\end{document}